\documentclass[11pt,a4]{article}

\title{Convergence of Proximal Policy Gradient Method for  Problems with Control Dependent Diffusion Coefficients}

\author{Ashley Davey\footnote{Department of Mathematics, Imperial College, London SW7 2BZ,\, UK. Email: ashley.davey18@imperial.ac.uk}\: and Harry Zheng\footnote{Department of Mathematics, Imperial College, London SW7 2BZ,\, UK. Email: h.zheng@imperial.ac.uk. Supported in part by the EPSRC (UK)  Grant (EP/V008331/1).}}
\date{\today}

\usepackage{graphicx}
\usepackage{float}
\usepackage{framed}
\usepackage{amsmath}
\usepackage{amssymb}
\usepackage{amsfonts}
\usepackage{mathrsfs}
\usepackage{array}
\usepackage{enumerate}
\usepackage{amsthm} 
\usepackage{bbm}
\usepackage{mathtools}
\usepackage{url}
\usepackage{standalone}
\usepackage{comment}
\usepackage{esdiff}
\usepackage{natbib}
\usepackage[margin = 1in]{geometry}
\usepackage[ruled]{algorithm2e}
\usepackage{caption}
\usepackage{subcaption}
\usepackage{emptypage}
\usepackage{xcolor}
\usepackage{hyperref}

\DeclareMathOperator*{\argmin}{arg\,min}

\providecommand{\R}{\mathbb{R}}
\providecommand{\N}{\mathbb{N}}

\providecommand{\D}{\mathcal{D}}
\providecommand{\p}{\mathbb{P}}

\providecommand{\F}{\mathcal{F}}
\providecommand{\A}{\mathcal{A}}

\providecommand{\U}{\mathcal{U}}
\providecommand{\X}{\mathcal{X}}
\providecommand{\Y}{\mathcal{Y}}
\providecommand{\B}{\mathcal{B}}
\providecommand{\LL}{\mathcal{L}}

\providecommand{\C}{\mathcal{C}}
\providecommand{\V}{\mathcal{V}}
\providecommand{\s}{\mathcal{S}}
\providecommand{\HH}{\mathcal{H}}
\providecommand{\Norm}{\mathcal{N}}
\providecommand{\Unif}{\text{\normalfont Unif}}

\providecommand{\E}{\mathbb{E}}

\providecommand{\ind}{\mathbbm{1}}
\providecommand{\half}{\ensuremath{\frac{1}{2}}}

\providecommand{\prox}{\text{\normalfont prox}}

\newcommand*{\defeq}{\mathrel{\vcenter{\baselineskip0.5ex \lineskiplimit0pt
                     \hbox{\scriptsize.}\hbox{\scriptsize.}}}%
                     =}

\newtheorem{theorem}{Theorem}[section] 
\newtheorem{lemma}[theorem]{Lemma}

\theoremstyle{plain} 
\newtheorem{remark}{Remark}

\newtheorem{assumption}{Assumption}
\newtheorem{definition}[theorem]{Definition}

\theoremstyle{definition} 
\newtheorem{example}[theorem]{Example}

\begin{document}

\date{}
\maketitle

\begin{abstract}
We prove convergence of the proximal policy gradient method for a class of constrained stochastic control problems with control in both the drift and diffusion of the state process. The problem requires either the running or terminal cost to be strongly convex, but  other terms may be non-convex. The inclusion of control-dependent diffusion introduces additional complexity in  regularity analysis of the associated backward stochastic differential equation. We provide sufficient conditions under which the control iterates converge linearly to the optimal control, by deriving representations and estimates of solutions to the adjoint backward stochastic differential equations. We introduce numerical algorithms that implement this method using deep learning and ordinary differential equation based techniques. These approaches enable high accuracy and scalability for stochastic control problems in higher dimensions. We provide numerical examples  to demonstrate the accuracy and validate the theoretical convergence guarantees of the algorithms. 
\end{abstract}

\smallskip


\smallskip


\section{Introduction}
Policy gradient methods (PGM) are widely used to solve continuous time stochastic control problems, particularly in reinforcement learning and optimal control. These methods iteratively adjust control policies using gradient-based updates, making them well-suited for high-dimensional problems where dynamic programming techniques may be intractable. One key challenge in applying PGM to continuous-time stochastic control is to ensure stability and convergence, particularly when constraints are present. The proximal policy gradient method (PPGM) extends PGM by incorporating regularisation, which stabilises updates and enables handling of non-smooth and control constrained  problems.

The study of PPGMs has gained significant attention recently.  \citep{reisinger2023linear} show pointwise linear convergence of the PPGM in terms of feedback control for problems with general nonlinear drift and cost functions.   Their approach relies on representing the PPGM update formula through the adjoint backward stochastic differential equation (BSDE). Establishing the regularity of the BSDE solution plays a crucial role in their convergence analysis. A key assumption in their model is that the diffusion coefficient of the state process is independent of control variable,  which simplifies the structure of the associated Hamilton-Jacobi-Bellman (HJB) equation into a semi-linear  partial differential equation (PDE).

The PPGM approach to the linear quadratic (LQ) problem has been studied in various forms, typically under unconstrained control.  \citep{fazel2018global} obtain convergence results for an LQ problem in discrete time with diffusion function independent of state and control. Motivated by the Riccati equation solution to the HJB equation, they search for linear policies, analysing a PGM-based sequence for the linear coefficient and prove convergence of the associated sequence of objective functions. \citep{hambly2023policy, wang2021global} extend this approach to mean-field games, including treatment in continuous time. In particular, \citep{wang2021global} determine convergence of a sequence of linear controls to the Nash equilibrium control. 
\citep{wang2020reinforcement} study reinforcement learning (RL) for continuous time stochastic control problems with both state and control in the diffusion, and introduce a relaxed framework in which the control policy itself is random. They show the optimal feedback control for the LQ problem is normally distributed, the mean of which is a linear function of the state. \citet{giegrich2024convergence} consider the combination of PGM and relaxed framework for unconstrained LQ problems,  where ordinary differential equation (ODE)-based solutions are used in deriving and proving convergence of a PGM-based sequence of linear coefficients for the mean of a normally distributed feedback control.

Despite these advances, the convergence of PPGM in problems with control-dependent diffusion coefficients, beyond the unconstrained LQ problem, remains largely an open question. The aim of this paper is to extend the results of \citep{reisinger2023linear} to more general control problems, specifically, problems where the control appears in the diffusion of the state process, and the running and terminal cost functions may be non-convex and may have unbounded derivatives. This additional complexity turns the HJB equation  into a fully nonlinear PDE, highly difficult to analyse and solve.

We derive sufficient conditions under which the PPGM  converges, where the state process is linear and there is at least one source of convexity in the running or terminal costs. Existing literature does not provide convergence results for this class of problems. The main difficulty arises from the complex dependence of the control process on the adjoint BSDE solutions as well as  the non-convexity of the objective functional with respect to control. The PPGM update formula in our setting depends on both elements $(Y, Z)$ of the solution to the adjoint BSDE. Establishing the regularity of $Z$ is challenging, but when the diffusion coefficient does not depend on the control, the update formula is independent of $Z$, making it unnecessary to establish its regularity as in  \citep{reisinger2023linear}. Standard approaches that rely on uniform ellipticity conditions or bounded derivatives of the value function no longer apply, requiring new techniques to establish well-posedness and convergence. We address this challenge by modifying the approach. Instead of proving pointwise convergence of iterate functions, we establish convergence in the space of square-integrable processes directly. This is achieved by exploiting the linearity of the state dynamics. Additionally, we demonstrate that allowing control in the diffusion can enhance the convergence properties of the PPGM, eliminating the need for convexity in the running cost. These considerations extend the applicability of PPGM beyond the scope of \citep{reisinger2023linear}.

We next compare our model to two related papers on iterative methods with control in the diffusion of the state process. 
\citep{kerimkulov2021modified} use successive approximations in which the Hamiltonian is directly minimised at each iteration step. In contrast, we only take a single step of gradient descent, which is numerically advantageous  as it reduces computation time by avoiding the optimisation of a potentially incorrect Hamiltonian at early iteration steps. \citep{kerimkulov2021modified} also perform convergence analysis with assumptions that running and terminal cost functions have bounded derivatives, excluding LQ problems, and obtain the convergence of the value function, whereas we establish the convergence of the control process itself. Their convergence proof uses a telescoping sum argument that ensures the objective function errors form a convergent series (hence converge to zero), circumventing the need for second-order regularity of the value function. This bypass is enabled by the exact Hamiltonian minimisation at each step. In contrast, such cancellation does not occur under PPGM, and their estimates do not carry over to our setting.
 \citep{zhou2023policy} show the convergence of the PGM for the objective function in an unconstrained minimisation problem with general state dynamics, under some strong  a priori assumptions, including that the value function is a classical solution to the HJB equation, the controlled state process stays in a bounded torus, the difference of the controls is bounded by that of the objective functions,  etc., in contrast, we only assume conditions on the model coefficients that  can be directly verified.

For implementation, we apply machine learning methods to approximate the control and BSDE processes. Compared to classical iterative methods, machine learning-based approaches provide a flexible and scalable way to approximate high-dimensional control processes, using neural networks to learn complex solution structures while maintaining computational efficiency. We use techniques similar to \citep{davey2022deep}, solving the control problem using the deep controlled second order BSDE (DC2BSDE) method that simulates the BSDE in the forward direction, driven by a deep neural network (DNN) process, where the terminal condition acts as a loss function to train the DNN parameters. The control process is also a DNN process, using the Hamiltonian (with approximate BSDE processes in place of the value function and its derivatives) as a loss function. We handle control constraints by approximating the projection onto the constraint set. We also propose an ODE-based method for the unconstrained LQ problem, similar to that in \citep{giegrich2024convergence}, and derive an explicit PGM that involves solving a series of linear ODEs  that converge to the solution of the Riccati equation.

The main contributions of this paper are as follows. We provide  
verifiable sufficient conditions under which the PPGM  converges linearly for linear controlled state processes with control dependent diffusion coefficients and general, possibly non-convex, objective functionals.
 We provide convergence analysis for this problem as it is not currently treated in the literature. The main challenge is to establish the  regularity for $(Y,Z)$ processes in the adjoint BSDE, as both appear in the update formula. Existing convergence results are lacking due to two key challenges: the absence of uniform ellipticity and the complex interaction between the control process and adjoint BSDEs. We address these difficulties by representing BSDE solutions using adjoint operators derived from the linear state dynamics (see Lemma \ref{lem_rep}). We implement an explicit ODE-based PGM for unconstrained LQ problems as an example and a DNN-based PPGM in the general setting. The method achieves high accuracy with low runtime while remaining stable across dimensions, circumventing the curse of dimensionality that limits traditional numerical schemes.

The remainder of this paper is outlined as follows. In Section \ref{sec_problem} we formulate the control problem, describe the proximal policy gradient method, and state the main result on the convergence of the PPGM to a fixed point (Theorem \ref{thm_convergence}). In Section \ref{sec_implement} we propose two numerical algorithms, one is a simplified implementation of the PPGM  for the unconstrained LQ problem involving a sequence of linear ODEs, and the other is a deep neural network implementation of the PPGM for general problems  involving two deep learning sub-algorithms approximating the BSDE solutions and corresponding control updates. In Section \ref{sec_numerics} we  present numerical examples, including nonconvex and constrained cases, showing accuracy and robustness of the algorithm. Section \ref{sec_conc} concludes.  Appendix contains all proofs.

\subsection{Notation} \label{sec_notation}

For two measure spaces $(A, \mu_A)$ and $\left(B, \mu_B\right)$ denote  $\B\left(A;B\right)$  the set of measurable functions from $A$ to $B$.
Let $\left(\Omega, \F,\left(\F_t\right)_{t \in [0,T]},\p\right)$ be a filtered probability space with fixed $T < \infty $.  Let $\left(E, |\cdot|\right)$ be a Euclidean space with inner product $\langle \cdot, \cdot \rangle_E$ and $\B_E$ the set of progressively measurable processes valued in $E$. For any $t \in [0,T]$ define 
$\s^2\left(t,T;E\right)  := \{ Y \in \B_E \colon \left\|Y\right\|_{\s^2} \defeq \E[\sup_{s \in [t, T]} |Y_s|^2]^{\frac{1}{2}} < \infty  \}$, $\HH^2\left(t,T;E\right)  := \{Z \in \B_E \colon \left\|Z \right\|_{\HH^2} \defeq \E[\int_t^T |Z_s|^2 ds]^{\frac{1}{2}} < \infty  \}$, 
and $\s^2\left(E\right)  := \s^2\left(0,T;E\right)$, $\HH^2\left(E\right) := \HH^2\left(0,T;E\right)$.
For $X, Y \in \HH^2(t, T; E)$ define 
$\langle X,Y \rangle_{\HH^2} \defeq \E[\int_t^T \langle X_s, Y_s\rangle_E ds]$.
For  $t \in [0,T]$ and $f, g \in \B( \Omega; E)$ define
$\left\langle f, g\right\rangle_{\LL^2}  \defeq \E\left[\langle f, g \rangle_E \right]$, $\left\|f\right\|_{\LL^2} = \langle f, f \rangle_{\LL^2}^{\frac{1}{2}}$,
$\LL^2(E, \F_t) = \{f \in \B(\Omega; E) \colon f \text{ is } \F_t\text{-meaurable, }\left\|f\right\|_{\LL^2}  < \infty\}$.
For a matrix $A$ and matrix valued process $B = (B_t)_{t \in [0,T]}$ define $A^\top$ as the transpose of $A$, $|A|   \defeq \sup_{|x| = 1} |Ax|$ and $\left\|B\right\|_\infty  \defeq \sup_{t \in [0,T]} |B_t|$.
For two square matrices $A, B$ with the same dimension, we say that $A \succeq B$ if the matrix $A - B$ is positive semi-definite.

\section{Problem Formulation and PPGM} \label{sec_problem}
Let $W=\left(W_t\right)_{t \in [0,T]}$ be a standard $1$-dimensional Brownian motion on the natural filtered probability space $(\Omega, \F,\left(\F_t)_{t \in [0,T]},\p\right)$ augmented with all $\p$-null sets, and $T < \infty$ a fixed finite horizon. We consider this one dimensional case for notational brevity, all results in the sequel extend to the multi-dimensional Brownian motion case. 
The state process $X^u$ satisfies a linear controlled stochastic differential equation (SDE), for $s \in [0,T]$,
\begin{align}\label{eq_state} 
dX^u_s  = \left(A_sX^u_s + B_s u_s\right)ds + \left(C_s X^u_s  + D_s u_s \right)dW_s, \  X^u_0  = x_0,
\end{align}
for some $x_0 \in \R^n$, where  $A,C \in \C^0\left([0,T];\R^{n \times n}\right)$, $B,D \in \C^0\left([0,T];\R^{n \times m}\right)$,  and $u$ is an  admissible control in the set  
$\U \defeq \{u\in  \HH^2(\R^m): u_s\in U, s\in[0,T], a.s.\}$
 for some nonempty closed convex set $U \subset \R^m$. 
  We have $X^u \in \s^2\left(\R^n\right)$ for all $u \in \U $ \citep[Theorem 3.2.2]{zhang2017backward}. The minimisation problem is the following 
 \begin{align}\label{obj}
 V \defeq \inf_{u \in \U} J(u) \defeq \E \left[ \int_0^T f_s\left(X^u_s, u_s\right)
ds  + g(X^u_T) \right], 
 \end{align}
 where  $f \colon [0,T] \times \R^n \times \R^m \to \R$ and  $g \colon \R^n \to \R$ are measurable, $f$ has the form 
 \[f_t(x, u) = f^1_t(x, u) + f^2_t(u), \] 
 and $f_t^1,f_t^2,g$ are continuously differentiable in $x, u$, their derivatives $\partial_u f^1_t, \partial_x f^1_t, \partial_u f^2_t, \triangledown g$ are Lipschitz continuous in $x, u$,  uniformly in $t$, and $\left(\partial_u f_t(0, 0)\right)_{t \in [0,T]} $ is square integrable. We write $f^1$ instead of $f^1_t$ and other functions if no confusion is caused in the paper.


In the following assumption we introduce a parameter $\mu > 0$ as a measure of the strength of convexity of the objective function. We will establish the convergence of PPGM  when  $\mu$ is sufficiently large, see Theorems \ref{thm_diff} and \ref{thm_convergence}. We consider two cases, where the strong convexity is in the running cost  (standard case) and in the terminal cost  (singular case). 
Recall that a function $h$ is $\mu$-strongly convex if the function $z \mapsto h(z) - \half \mu |z|^2$ is convex. 

\begin{assumption} \label{ass_strong}
There exists a sufficiently large $\mu > 0$ such that one of the following holds.
\begin{enumerate}
\item Standard case:  $f^2$ is $\mu$-strongly convex in $u$, uniformly in $t$. 
\item Singular case: $g$ is $\mu$-strongly convex. Furthermore,
define, for $t \in [0,T]$,
$\A_t  = A_t + A_t^\top + C_t^\top C_t$, $\B_t  = B_t + C_t^\top D_t$, and  $\D_t  = D_t^\top D_t$.
Then one of the following holds.
\begin{enumerate}[i)]
\item $\A_t$ is positive definite for all $t$, and there exists $\delta > 0$ such that 
$\D_t - \B_t^\top \A^{-1}_t \B_t \succeq \delta \ind_m. $
\item $\A_t$ is positive semi-definite for all $t$,  there exists $\delta > 0$ such that 
$ \D_t - \B_t^\top \B_t \succeq \delta \ind_m$, and 
$T$ is sufficiently small (the condition on $T$ is not required if $\B_t=0$ for all $t$). 
\end{enumerate}
\end{enumerate}
\end{assumption}



We use the following result in establishing the convergence of the PPGM   to a stationary point.

\begin{theorem}[Stationary point characterisation, \citep{reisinger2023linear} Theorem 3.10] \label{thm_stationary} Let $\left(X, \|\cdot\|_X\right)$ be a Hilbert space, $F \colon X \to \R$ be a differentiable function and $\tilde{F} \colon X \to \R \cup \{\infty\}$ be a proper, lower semicontinuous and convex function. Let $x^* \in X$ such that $\tilde{F}\left(x^*\right) < \infty$. Then $x^*$ is a stationary point of $F+\tilde{F}$, that is,
$$\liminf_{x \to x^*} \frac{(F+\tilde{F})(x) - (F+\tilde{F})(x^*)}{\|x-x^*\|_X} \geq 0, 
$$
 if and only if there exists $\bar{\tau} > 0$ such that for all $\tau \in (0, \bar{\tau})$
\begin{align*}
x^* = \prox_{\tau \tilde{F}}\left(x^* - \tau \triangledown F\left(x^*\right)\right),
\end{align*}
where $\triangledown F$ denotes the Fréchet derivative of $F$ and, for all $x \in X$, 
\[\prox_{\tau \tilde{F}}\left(x\right) = \arg \min_{z \in  X}\left(\half \|z-x\|_X +\tau \tilde{F}\left(z\right)\right).\]
\end{theorem}

In the context of this paper, $X=\HH^2(\R^m)$, $F=J$ in (\ref{obj}), and $\tilde F(u)=0$ if $u_t\in U$ a.e. for $t\in [0,T]$ and $\infty$ otherwise (an infinite penalty function for violation of control constraints). To find the stationary control, we use the  Hamiltonian $H\colon [0,T] \times \R^n \times \R^m \times \R^n \times \R^{n} \to \R$, defined by
\begin{align}\label{eq_def_H}
H_t\left(x, u, y, z\right) & \defeq y^\top \left(A_tx + B_t u\right) +  z^\top \left(C_t x + D_t u\right) + f_t(x, u),
\end{align}
for $t \in [0,T]$, $u \in \R^m$, and $x, y, z \in \R^n$, where $a^\top b$ denotes the inner product of $a, b \in \R^d$ for any $d \in \N$. It can be shown (see e.g. \citep[Lemma 3.1]{acciaio2019extended}) that under the assumptions on $f$ and $g$ and their derivatives, $J$ is Fr\'{e}chet differentiable with $\triangledown J(u) \in \HH^2(\R^m)$, given by, for $t \in [0,T]$,
\[(\triangledown J(u))_t = \partial_u H_t\left({X}_t^u, u_t, {Y}_t^u, {Z}_t^u\right),\]
where $X^u$ is the solution to SDE \eqref{eq_state} and $(Y^u, Z^u)$ satisfy 
 the following adjoint BSDE, for $s \in [0,T]$,
\begin{align} \label{eq_adjoint}
dY^u_s & = -\left(A_s^\top Y^u_s + C_s^\top Z^u_s + \partial_x f_s(X^u_s, u_s)\right)ds + Z_s^u dW_s, \
Y^u_T = \triangledown g(X^u_T).
\end{align}
The HJB equation for this problem can be written as
\begin{align} \label{eq_pde}
\partial_t v(t, x) + \inf_{u \in U}   H_t\left(x, u, \partial_x v(t, x), \half \partial_{xx}v(t, x)  \left(C_t x + D_t u \right)\right)  = 0, 
\end{align}
with $v(T, x) = g(x)$. The policy gradient method (PGM) uses gradient descent with respect to the Hamiltonian. In the HJB equation the Hamiltonian takes the first and second derivatives of the value function - if it is indeed differentiable - as input. We do not know these a priori, and  instead input solutions to a BSDE, which act as estimates for these quantities. To deal with the constraint function, we introduce the proximal map $\text{prox}_{U} \colon \R^m \to \R^m$ as
\begin{align*}
\text{prox}_{U}\left(u\right) \defeq \argmin_{z \in U} \left\{\half |z-u|^2\right\}, \ u \in \R^m. 
\end{align*}
 The algorithm seeks feedback controls $\phi \colon [0,T] \times \R^n \to U$. Let $\B\left([0,T] \times \R^n ;U\right)$ be the space of such measurable functions. 
Then define
\[\V \defeq \left\{ \phi \in \B\left([0,T] \times \R^n ;\R^m\right) \colon (\phi_s(X^\phi_s))_{s \in [0,T]} \in \U  \right\},\]
where $X^\phi$ is a solution to SDE (\ref{eq_state}) with closed-loop control $u_s=\phi_s(X^\phi_s)$ for $ s \in [0,T]$.  
For some arbitrary initial guess $\phi_0$, we can define an iterative scheme as follows, for all $k \in \N$ and  $\left(t,x\right) \in [0,T] \times \R^n$.
\begin{definition}[Proximal policy gradient method]
Define the updating function $\LL \colon \B\left([0,T] \times \R^n ;U\right) \to \B\left([0,T] \times \R^n ;U\right)$ by
\begin{align} \label{eq_L}
\LL\left(\phi\right)_t\left(x\right) \defeq \prox_{U}\left(\phi_t\left(x\right) - \tau \partial_u H_t\left( x, \phi_t\left(x\right), Y^{t, x, \phi}_t, Z^{t, x, \phi}_t\right) \right), 
\end{align}
for $\phi \in \V$, $t \in [0,T]$ and $x \in \R^n$, where  $\left(X^{t, x, \phi}, \, Y^{t, x, \phi}, \, Z^{t, x, \phi}\right)$ are solutions to the following FBSDE, for $s \in [t, T]$, 
\begin{align} \begin{split}
dX^{t,x,\phi}_s & = \left(A_sX^{t,x,\phi}_s + B_s u_s\right) ds + \left(C_sX^{t,x,\phi}_s + D_s u_s\right) dW_s, \
X^{t,x,\phi}_t  = x, \\
dY^{t,x,\phi}_s & = -\left(A_s^\top Y^{t,x,\phi}_s + C_s^\top Z^{t,x,\phi}_s + \partial_x f_s(X^{t,x,\phi}_s, u_s)\right)ds + Z_s^{t,x,\phi} dW_s, \
Y^{t,x,\phi}_T = \triangledown g(X^{t, x, \phi}_T)
\end{split} \label{bsde}
\end{align}
where $u_s=\phi_s\left(X^{t,x,\phi}_s\right)$. 
Fix $\phi^0 \in \V$, then for $k \in \N$ define the PPGM scheme $\left(\phi^k\right)_{k \in \N}$ by
\begin{align}\label{update}
\phi^{k+1}_t\left(x\right) = \LL(\phi^k)_t(x),
\end{align}
for $t \in [0,T]$ and $x \in \R^n$.
\end{definition}

For $\phi \in \V$ we define the associated control $u^{\phi} \in \U$ by $u^{\phi}_t = \phi_t\left(X^\phi_t\right)$. Let $u^k \defeq u^{\phi^k}$ be the corresponding sequence of controls. By the Markov property, the sequence $(u^k)_{k \in \N}$ satisfies the relation
\begin{align} \label{update_control}
u^{k+1}_t 
& = \prox_{U}\left(u^k_t - \tau \partial_u H_t\left(X^{k}_t, u^k_t, Y^{k}_t, Z^{k}_t\right) \right),
\end{align}
where $(X^k, Y^{k}, Z^{k}) \defeq (X^{0, x_0, \phi^k}, Y^{0, x_0, \phi^k}, Z^{0, x_0, \phi^k})$ are the corresponding solutions of \eqref{bsde}.
The sequence $u^k$ is then the natural PPGM sequence for the function $J$ under the control constraint. To show convergence of the algorithm, we  focus on the iterates $u^k$, but in a practical algorithm these cannot be found directly, and we instead focus on the iterates $\phi^k$, with additional approximations by neural networks.

\begin{remark}
In \citep{reisinger2023linear} the PPGM is used for control problems without control in the diffusion, leading to an update formula of the form
\begin{align*}
u^{k+1}_t
& = \prox_{U}\left(u^k_t - \tau \partial_u H_t\left(X^{k}_t, u^k_t, Y^{k}_t\right) \right),\ k \in \N,
\end{align*}
which depends on the BSDE solution $(Y^{k}, Z^{k})$ only through $Y$. The treatment of $Y$ and its dependence on the forward process $X$ is standard in the literature, see \citep{zhang2017backward}. In our approach we also need to deal with the less regular process $Z$, leading to technical difficulties to overcome. Typically, $Z$ is determined via applying the martingale representation theorem to $Y$, in a nonconstructive manner, and its regularity is difficult to estimate. We do not directly address the regularity of $Z$, instead, we give a representation of  $\partial_u H_t\left(X^{k}_t, u^k_t, Y^{k}_t, Z^{k}_t\right) $ in terms of the adjoint operators of the associated stochastic integrals  for SDE (\ref{eq_state}). This representation is  crucial in our treatment of control in diffusion coefficients, as it yields the necessary regularity estimates that facilitate convergence analysis in the general case.
\end{remark}

We show in the next theorem that differences of control iterates at stage $k \in \N$ can be bounded by those at stage $k - 1$, yielding an iterative bound.

\begin{theorem}  \label{thm_diff}
Let Assumption  \ref{ass_strong}  hold.  Then there exists a constant $K >0$, independent of $\mu$, such that for sufficiently large $\mu>0$ and sufficiently small $\tau>0$, the sequence $(u^k)_{k \in \N}$ satisfies
\begin{align*}
& \left\|u^{k+1} - u^k\right\|_{\HH^2} \leq \left(1-\tau\left( \rho \mu - K\right)\right) \left\|u^k - u^{k-1}\right\|_{\HH^2},
\end{align*}
for all $k \in \N$, 
where $\rho = \frac{1}{2}$ in the standard case and $\rho = \frac{\delta}{2}$ in the singular case.
\end{theorem}

 \begin{remark} In Theorem \ref{thm_diff} and Assumption \ref{ass_strong}, we assume $\mu$ is sufficiently large, 
 that is, $\mu> K/\rho$ for some constant $K$ that is determined by model parameters but is independent of $\mu$, see proofs of  Theorems \ref{thm_diff} and \ref{thm_convergence}, which ensures $\hat{c} \defeq 1-\tau\left( \rho \mu - K\right)$ is less than 1. We also require $\tau$  sufficiently small to ensure $\hat{c} $ is greater than 0. Note that the range of $\tau$  is determined by $\mu$ as well as other model parameters. 
\end{remark}

Theorem \ref{thm_diff}
yields a bounding sequence of the form $\left\|u^{k+1} - u^k\right\|_{\HH^2} \leq \hat{c}^k \left\|u^{1} - u^0\right\|_{\HH^2} $
 for $\hat{c}\in (0,1)$, which shows the sequence $(u^k)$ is a Cauchy sequence and converges to 
 a fixed point $u^*$, that is a stationary point of the control problem. The constant $K$ in Theorem \ref{thm_diff} is dependent on $D$ so we cannot freely assume $\delta$ is arbitrarily large to ensure a contraction mapping  under the singular case. We now state the main result of the paper. 

\begin{theorem} \label{thm_convergence}
Let Assumption  \ref{ass_strong} hold.  Then for sufficiently large $\mu>0$ and sufficiently small $\tau>0$, there exists $u^*\in \U$ and $\hat{c} \in [0,1)$ such that
\begin{align}\label{eq_conv}
\left\|u^{k} - u^{*}\right\|_{\HH^2} \leq \hat{c}^k\left\|u^{0} - u^{*}\right\|_{\HH^2},
\end{align}
for all $k \in \N$. 
Furthermore, $u^{*}$ is a stationary point of $J$, satisfying the control constraint. 
\end{theorem}

\section{Implementation}\label{sec_implement}
In this section, we outline the implementation of the theoretical method using alternating numerical schemes. In these schemes, each iteration alternates between solving the BSDE (\ref{bsde}) and updating the control (\ref{update}) based on the obtained solution. For unconstrained LQ problems, this is achieved through an explicit ODE-based update rule, while for the general constrained setting, a deep learning-based method uses neural networks to approximate both the control and BSDE processes.

\subsection{Unconstrained LQ Problem}

The LQ problem has been extensively studied in the literature, see \citep{sun2016open} for exposition. 
In the unconstrained setting, we have  $U = \R^m$ and 
\begin{align}
f_s(x, u) \defeq \half x^\top Q_s x + x^\top S_s^\top  u  + \half u^\top R_s u,\quad
g(x)=x^\top G x,
\label{eq_value_lq}
\end{align}
where  coefficients $Q \in \C^0\left([0,T];\R^{n \times n}\right)$, $R \in \C^0\left([0,T];\R^{m \times m}\right)$, $S \in \C^0\left([0,T];\R^{m \times n}\right)$, $G \in \R^{n \times n}$ are deterministic, and $R$, $Q$ and $G$ are symmetric matrices. To ensure Assumption \ref{ass_strong} is satisfied, we assume there exists $\mu > 0$ such that $R \succeq \mu \ind_m$ or $G \succeq \mu \ind_n$, the latter paired with the corresponding singular condition on $A, B, C$ and $D$. 

The update formula (\ref{eq_L})  becomes
\begin{align}
\LL\left(\phi\right)_t\left(x\right) = \phi_t\left(x\right) - \tau \left(B_t^\top Y^{t,x,\phi}_t + D_t^\top Z^{t,x,\phi}_t + R_t \phi_t\left(x\right) + S_t x \right), 
\label{eq_L_u}
\end{align}
for $\phi \in \V$, $t \in [0,T]$ and $x \in \R^n$, where  $\left(X^{t, x, \phi}, \, Y^{t, x, \phi}, \, Z^{t, x, \phi}\right) $ are solutions to the FBSDE (\ref{bsde}). 
\begin{theorem} \label{thm_linear}
Define a function $\phi \in \V$ by
$\phi_t\left(x\right) \defeq \alpha_tx $
for $\alpha \in \C^0\left([0,T];\R^{m \times n}\right)$. Then there exists $\tilde{\alpha} \in \C^0\left([0,T];\R^{m \times n}\right)$ and $a \in \C^0\left([0,T];\R^{n \times n}\right)$ such that
$\LL\left(\phi\right)_t\left(x\right) = \tilde{\alpha}_tx$, $Y^{t,x,\phi}_t = a_t x$ and $Z^{t,x,\phi^k}_t  = a_t \left(C_t + D_t \tilde{\alpha}_t\right) x$
for all $t \in [0,T]$ and $x \in \R^n$, and $a$ solves the following first order linear inhomogeneous ODE
\begin{align} \label{update_a}
\dot{a}_s + a_s \left(A_s + B_s \alpha_s\right) + A_s^\top a_s + C_s^\top a_s \left(C_s + D_s \alpha_s\right) +  Q_s + S_s^\top \alpha_s & = 0,
\end{align}
for $s \in [t,T]$ with $a_T = G$.
\end{theorem}

Note that the ODE governing $a$ is independent of the choice of $t$, so can be extended to solutions on the entire horizon $[0,T]$. We have shown that the update (\ref{eq_L_u}) maps $\V_L \defeq \{\phi \in \V \colon \phi \text{ linear}\}$ to itself, and can be written as the following, for $t \in [0,T]$, $x \in \R^n$ and $k \in \N$,
\begin{align}\begin{split}
\phi^k_t\left(x\right) & = \alpha^k_t x, \qquad
Y^{t,x, \phi^k}_t = a_t^k x , \qquad
Z^{t,x,\phi^k}_t  = a^k_t \left(C_t + D_t \alpha^k_t\right) x ,
\end{split}\label{eq_notation} 
\end{align}
where $a^k \in \C^1\left([0,T];\R^{n \times n}\right)$ solves (\ref{update_a}) with $\alpha = \alpha^k$, and $\alpha^k \in \C^0\left([0,T];\R^{m \times n}\right)$ satisfies
\begin{align}
\alpha^{k+1}_t & =  \alpha^k_t    - \tau \left(B_t^\top a_t^k   + D_t^\top a^k_t   \left(C_t + D_t \alpha^k_t   \right) + R_t \alpha^k_t     + S_t\right).
\label{update_alpha}
\end{align}
Given some initialised $\alpha^0 \in \C^0\left([0,T]; \R^{m \times n}\right)$  we can implement (\ref{update_alpha}) to get successive controls $\phi^k$, $k \in \N$. 
We solve the ODE (\ref{update_a}) using an ODE solver in Python. We discretise the time space $[0,T]$ into a partition $\left(t_i\right)_{i=0}^N$ for some $N \in \N$ and output the optimal control evaluated at these time points. In the implementation we  take $t_i = i\frac{T}{N}$. For the ODE solver, we need a function to evaluate over the entire interval, not just discrete points, so we input a piecewise constant approximation of $\alpha^k$. 
The algorithm  performs this update until a convergence criterion is reached. This is either a maximum number of iterations, or if the algorithm does not improve enough over a single iteration. We deem the algorithm to converge if the relative difference of $\alpha$ between iterations is sufficiently small, given for $ k \in \N$ as
$
\Delta_k \defeq {\|\alpha^k - \alpha^{k-1}\|_\infty } / {\|\alpha^{k-1}\|_\infty}. 
$

\begin{algorithm}[H] \label{alg_PPGM_uncon} 
\SetAlgoLined
\KwResult{Optimal feedback control}
 Initialise $k = 0$ and tolerance $\epsilon> 0$\;
 \For{$i = 0,1 ,2 ,\ldots, N-1$}{ 
 Initialise $\alpha^0_i \sim \text{Unif}\left([-0.1,0.1]^{m \times n}\right)$
 }
\While{$\alpha^k$ \text{\normalfont not converged} ($\Delta_k \geq \epsilon$)}{
	Solve ODEs (\ref{update_a}) with $\alpha = \alpha^k$ to output $a^k_i = a^k_{t_i}$ for $i = 0,1 ,2 ,\ldots, N-1$\;
	Generate $\alpha^{k+1}_i$ using (\ref{update_alpha}) with $t = t_i$ for $i = 0,1 ,2 ,\ldots, N-1$\;
	Set $k = k+1$\;
	}
 Output control $x \mapsto \alpha^k_i x$ for $i = 0,1 ,2 ,\ldots, N-1$\;
  \caption{Unconstrained Policy Gradient Method for LQ Problem (LQ-PGM)}
\end{algorithm}

\subsection{General Algorithm}

We consider the numerical application of the PPGM scheme (\ref{update}). We cannot directly implement the method, as that would require exact solutions of BSDEs, which are unavailable.  Instead, we use a variant of the DC2BSDE method \citep{davey2022deep} to use approximating functions at each stage. In the unconstrained setting, our approach is equivalent to the DC2BSDE method without value function approximation. With constraints, we use a projection term, rather than the penalty function used in \citep{davey2022deep}. The value function is then approximated by running Monte Carlo simulations.

To facilitate numerical algorithms, we use a finite dimensional subspace to determine the feedback control $\phi$ and function $(t, x) \mapsto (Y^{0, x, \phi}_0, Z^{0, x, \phi}_t)$. We  use neural networks due to their high approximation capabilities \citep{leshno1993multilayer}.  The approximate function has the form
$S(t, x; \theta) = \Psi_3 \circ h \circ \Psi_2 \circ h \circ \Psi_1 (t,x)$, 
 where  $\Psi_i(z) = A_i z + b_i$  is a linear ``layer'' of some dimension $d_i$ with $d_0 = 1 + d, d_3 = k$ and $d_1, d_2$ some predetermined number of ``hidden nodes'', 
 and $h$ is some element-wise non-polynomial activation function, e.g. $h(z) = \tanh(z)$. In this case the parameter vector $\theta = (A_{i}, b_i)_{i = 1}^3 \in \R^\rho$ has dimension $\rho = \sum_{i=1}^3 d_i (d_{i-1} + 1)$. We fix an architecture of 2 hidden layers, but this can be naturally extended.

\subsubsection{Deep Learning for Solving BSDEs}

For fixed $\phi \colon [0,T] \times \R^n \to \R^m$ we can use step 1 of the DC2BSDE algorithm \citep{davey2022deep} to determine pathwise solutions $(X^j_i, Y^j_i, Z^j_i)$, $i = 1, \ldots, N, j = 1, \ldots M$, 
for some time discretisation ($0 = t_0 < \ldots < t_N = T$) of size $N$ and batch size $M$. To do this, we define a neural network $Z(t, x; \theta)$ and a start network $y(x; \kappa)$ corresponding to $Y_0$, depending on some parameters $\theta$ and $\kappa$. We start $X$ randomly in some set $\X \subset \R^n$, typically a Cartesian product of intervals. Define a loss function
\[\LL^\phi(\theta, \kappa) = \frac{1}{M}\sum_{j = 1}^M \left| Y^j_N - \triangledown g(X^j_N)\right|^2, \]
where, for $j = 1, \ldots, M$ and $i = 0, \ldots, N - 1$, $X^j_0  = x^j$, 
$Y^j_0  = y(x^j; \kappa)$, and 
\begin{align*}
X^j_{i+1} & = X^j_{i} + \left(A_i X^j_{i} + B_i \phi\left(t_i, X^j_{i}\right)\right) (t_{i+1} - t_i) + \left(C_i X^j_{i} + D_i \phi\left(t_i, X^j_{i}\right)\right) \Delta W^j_i, \\
Y^j_{i+1} & = Y^j_{i} -\left(A_i^\top Y^j_{i+1} + C_i^\top Z(t_i, X^j_{i}; \theta) + \partial_x f_i\left(X^j_{i}, \phi\left({t_i}, X^j_{i}\right)\right)\right)(t_{i+1} - t_i) + Z(t_i, X^j_{i}; \theta)\Delta W^j_i,
\end{align*}
where we use the shorthand $A_i = A_{t_i}$, and we generate $x^j \sim \Unif(\X)$ and $\Delta W^j_i \sim \Norm(0,t_{i+1} - t_i)$. 
The DBSDE algorithm is then, for $l \in \N$
\begin{align} \label{eq_update_dbsde} \begin{split}
\kappa^{l+1} & = \kappa^l - \delta \partial_\kappa \LL^\phi(\theta^l, \kappa^l), \\
\theta^{l+1} & = \theta^l - \delta \partial_\theta \LL^\phi(\theta^l, \kappa^l), \end{split}
\end{align}
where the operation for $\theta$ and $\kappa$ are element-wise.

\subsubsection{Deep Learning for Control Iteration}
Now suppose we have a control $\phi^k \colon [0,T] \times \R^n \to \R^m$ of neural network form $\phi^k(t, x) = \phi(t, x; \chi^k)$ for some parameter set $\chi^k$. Assuming convergence of the BSDE algorithm, We can find corresponding $\kappa^k \defeq \kappa^{\phi^k}$ and $\theta^k \defeq \theta^{\phi^k}$, and outputted simulation points $\left(X^j_i, Y^j_i, Z\left(t_i, X_i^j; \theta^k\right)\right) $. We want to update
\begin{align*}
\tilde{\phi}_{i}^{k+1, j} \defeq \prox_{U}\left(\phi^k\left(t_i, X^j_i\right) - \tau \left(B_{i}^\top Y^j_i + D_{i}^\top Z\left(t_i, X_i^j; \theta^k\right) + \partial_u f_{i}\left(X^j_i, \phi^k\left(t_i, X^j_i\right)\right)\right)\right).
\end{align*}
However this does not give us a function to use in the next stage of simulation. We therefore need to find a functional approximation of $\tilde{\phi}_{i}^{k+1, j}$ given input data $\left(t_i, X^j_i\right)$ for each $i, j$.  We again use a DNN for this, defining $\phi\left(t, x; \nu \right)$ as a neural network for some parameter set $\nu$.
Define the second loss function
\begin{align*}
\mathcal{G}(\nu, \chi^k) & = \frac{1}{BN}\sum_{j = 1}^M \sum_{i = 0}^{N-1} \bigg| \phi\left(t_i, X^j_i; \nu\right)- \tilde{\phi}_{i}^{k+1, j}\bigg|^2 .
\end{align*}
For some $\rho > 0$ apply the gradient descent algorithm
\begin{align} \label{eq_update_ppgm}
\nu^{l+1} & = \nu^l - \rho \partial_{\nu} \mathcal{G}\left(\nu^l, \chi^k\right),
\end{align}
until convergence to some $\nu^*$. We then denote $\chi^{k+1} = \nu^*$, and use this to define $\phi^{k+1}$, and repeat the process. We can repeat this algorithm until a tolerance is reached. Let $u^k$ be the control associated to $\phi(\cdot, \cdot; \chi^k)$. Based on Theorem \ref{thm_convergence}, define
$\Delta_k = {\|u^{k} - u^{k-1}\|_{\HH^2}}/{\|u^{k-1}\|_{\HH^2}}. $
In practise, we initialise the neural networks at each iteration by the outputted networks from the previous iteration, as these would pose as the ``best initial guess'' for the neural network optimisation as the learning rates $\tau, \delta, \rho$ are small. In the unconstrained case, this algorithm reduces to the DC2BSDE algorithm of \citep{davey2022deep}. Note that this algorithm will introduce errors that are not accounted for in our analysis, due to the approximations of BSDE solutions and the corresponding control function. 

\begin{algorithm}[H] \label{alg_PPGM_deep} 
\SetAlgoLined
\KwResult{Optimal feedback control}
 Initialise $k = 0$ and tolerances $\epsilon_1, \epsilon_2, \epsilon_3 > 0$ or max substep counts $k_1, l_2, l_3 \in \N$\;
 Initialise $\hat{\chi}^k, \hat{\theta}^k, \hat{\kappa}^k$ parameters each normally distributed with mean 0 and a small variance\;
\While{$\hat{\chi}^k$ \text{\normalfont not converged} ($\Delta_k \geq \epsilon_1$ or $k < k_1$)}{
Set $l = 0$\;
Initialise $\theta^l = \hat{\theta}^k, \kappa^l = \hat{\kappa}^k$\;
\While{$\theta^l, \kappa^l$ \text{\normalfont not converged} ($\LL^{\phi^k}(\theta^l, \kappa^l) \geq \epsilon_2$ or $l < l_2$)}{
Compute $\theta^{l+1}, \kappa^{l+1}$ using (\ref{eq_update_dbsde})\;
Set $l = l+1$\;
}
Set $\hat{\theta}^{k+1} = \theta^l, \hat{\kappa}^{k+1} = \kappa^l$\;
Set $l = 0$\;
Initialise $\nu^l = \hat{\chi}^k$\;
\While{$\nu^l$ \text{\normalfont not converged} ($\mathcal{G}(\nu^l, \hat{\chi}^k) \geq \epsilon_3$ or $l < l_3$)}{
Compute $\nu^{l+1}$ using (\ref{eq_update_ppgm})\;
Set $l = l+1$\;
}
Set $\hat{\chi}^{k+1} = \nu^l$\;
Set $k = k+1$\;
}
 Output control $\phi^k(\cdot; \hat{\chi}^k)$\;
  \caption{Proximal Policy Gradient Method (PPGM)}
\end{algorithm}

\section{Numerics} \label{sec_numerics}

In this section we apply the PPGM algorithms  to some numerical examples. We run the methods for a maximum of $k_1=200$ steps, with $\tau$ sufficiently small as in the proof of Theorem \ref{thm_diff} on a problem-by-problem basis. The LQ-PGM and PPGM algorithms also stop if the respective convergence criteria satisfy $\Delta_k \geq 10^{-6}$ and $\Delta_k \geq 10^{-5}$ respectively. The DNN sub-steps are run for $l_2 = l_3 = 100$ iterations, with learning rate $\rho = \delta = 0.01$, batch size $M = 50$ and number of time steps $N = 10$. The neural networks have 2 hidden layers of size $d_1 = d_2 = 10$, with $\tanh$ activation function, and the first layer is normalised. They are implemented within \texttt{tensorflow} with default initialisation for parameters. We generate sample points over $\X = [-10, 10]$, and when evaluating $\HH^2$ error given a control function $\phi$, we simulate $X^\phi$ using $100$ time steps and a sample size of $10000$. The code used to generate these numerics is available at \url{https://github.com/Ashley-Davey/PPGM}.

\subsection{LQ Problems}

We consider first the LQ problem (\ref{eq_value_lq}) with varying dimensions and control constraints, with various impacts on the convexity of the problem. Data (coefficients) are generated randomly. 
 The value function and optimal control are given by $v(t, x) = x^\top a^*_t x$ and $\phi^*_t(x) = \alpha^*_t x$, for some coefficients found by the Riccati equation, see e.g. \citep{sun2016open}.

\subsubsection{Convex problem with standard assumptions}

Data used are $n = 2, m = 3$, $T = 1$,   $S=0$, $G=\ind_2$ and 
$$A  = \begin{pmatrix}
0.041 & 0.11 \\
-0.25 & 0.099
\end{pmatrix}, \
B  = \begin{pmatrix}
-0.177 & -0.204 & -0.157 \\
0.077 & 0.052 &  0.019
\end{pmatrix}, \
C  = \begin{pmatrix}
0.04 & 0.093 \\
-0.148 & 0.189
\end{pmatrix},
$$
$$ 
D  = \begin{pmatrix}
-0.236 & 0.085 & 0.041 \\
0.029 & -0.180 & -0.151 
\end{pmatrix}, \
Q= \begin{pmatrix}
0.2 & 0.2  \\
0.2 & 0.2 
\end{pmatrix},\
R=\begin{pmatrix}
 1.217 & 0.019 & -0.236 \\
 0.019 & 0.809 & 0.086 \\
 -0.236 & 0.086 & 1.264
\end{pmatrix}.
$$
This problem is convex  and satisfies the standard assumptions with $\mu = 0.79$.  Figure \ref{fig_ex1_h2}  (a) shows the $\HH^2$ error decreasing as iteration steps increase. The general PPGM is less accurate than the ODE-based LQ-PGM method due to the additional neural network approximation error.  Figure \ref{fig_ex1_h2}  (b) shows the losses associated with the PPGM algorithm, which decrease with iteration step, indicating the algorithm is improving at every step. Figure \ref{fig_ex1_h2}  (c) displays the value function in terms of (the first dimension of) $x$ at time 0. Even though the $\HH^2$ accuracy is significantly different between the two algorithms, they both recover the linear control and quadratic value function  well.

\begin{figure}[H]
\centering
\begin{minipage}{.33\textwidth}
\centering
\begin{subfigure}[b]{\textwidth}
\includegraphics[width=\textwidth]{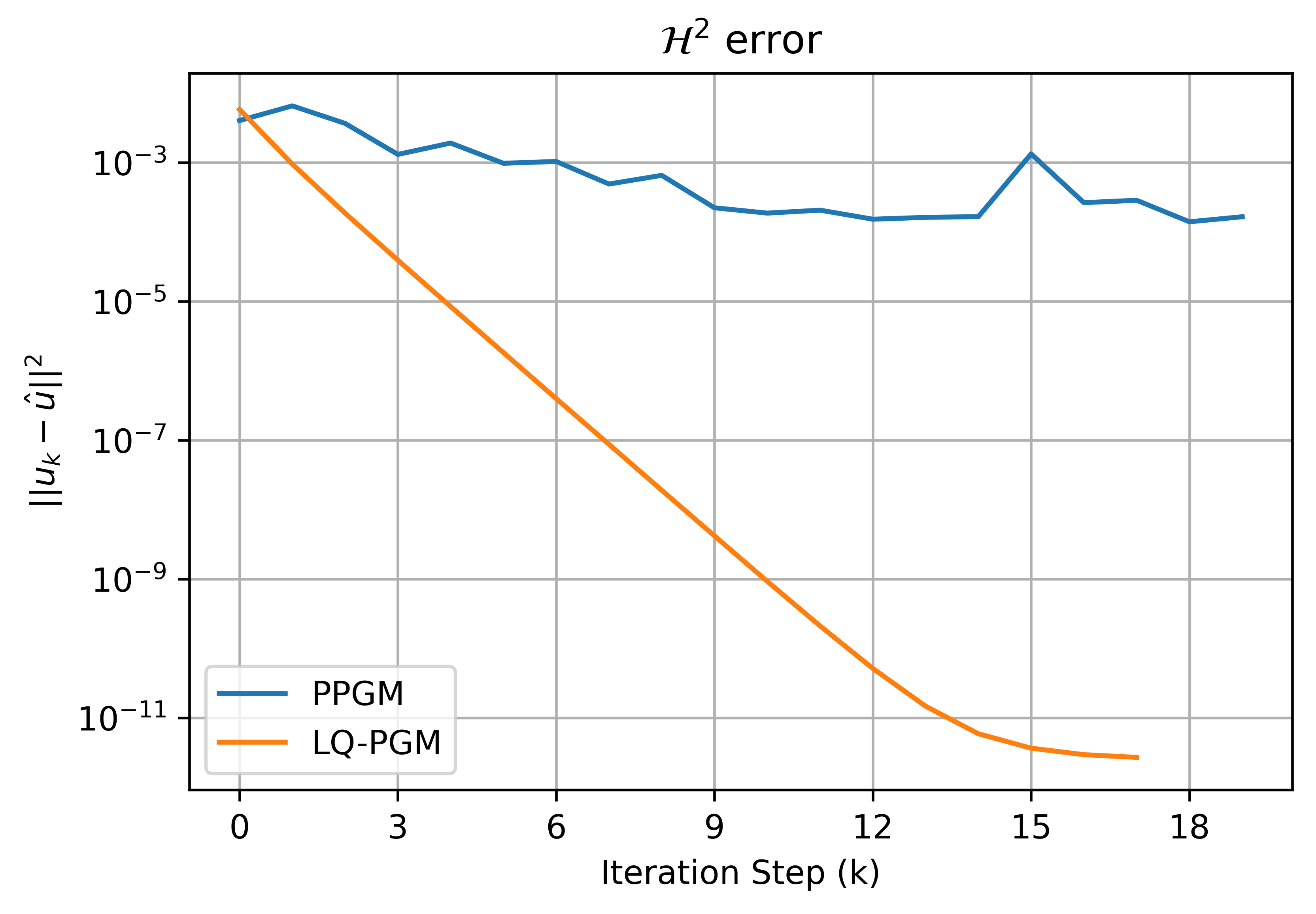}
\caption{Control error.}
\end{subfigure}
\end{minipage}%
\begin{minipage}{.33\textwidth}
\centering
\centering
\begin{subfigure}[b]{\textwidth}
\includegraphics[width=\textwidth]{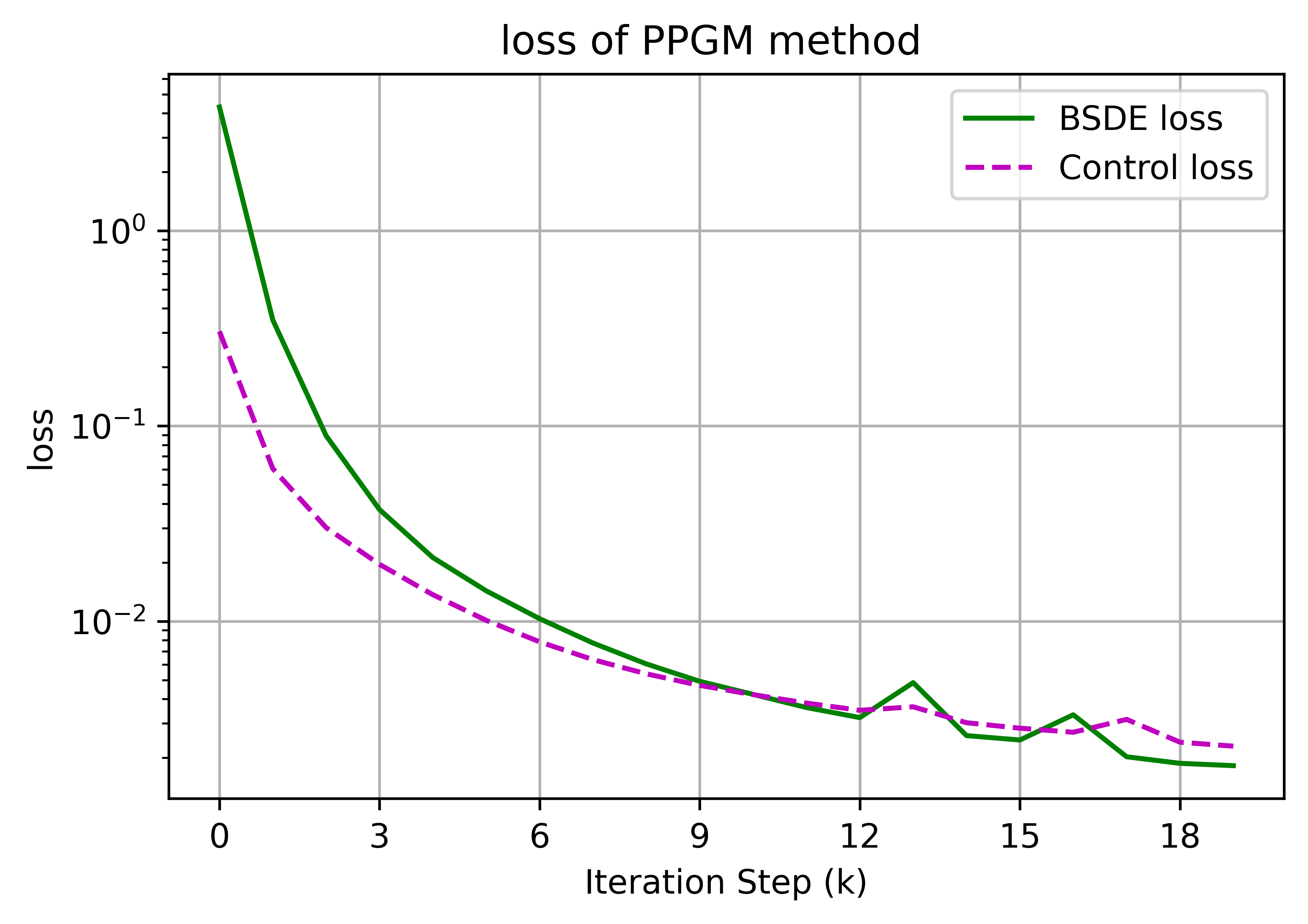}
\caption{Loss functions for PPGM}
\end{subfigure}
\end{minipage}
\begin{minipage}{.33\textwidth}
\centering
\centering
\begin{subfigure}[b]{\textwidth}
\includegraphics[width=\textwidth]{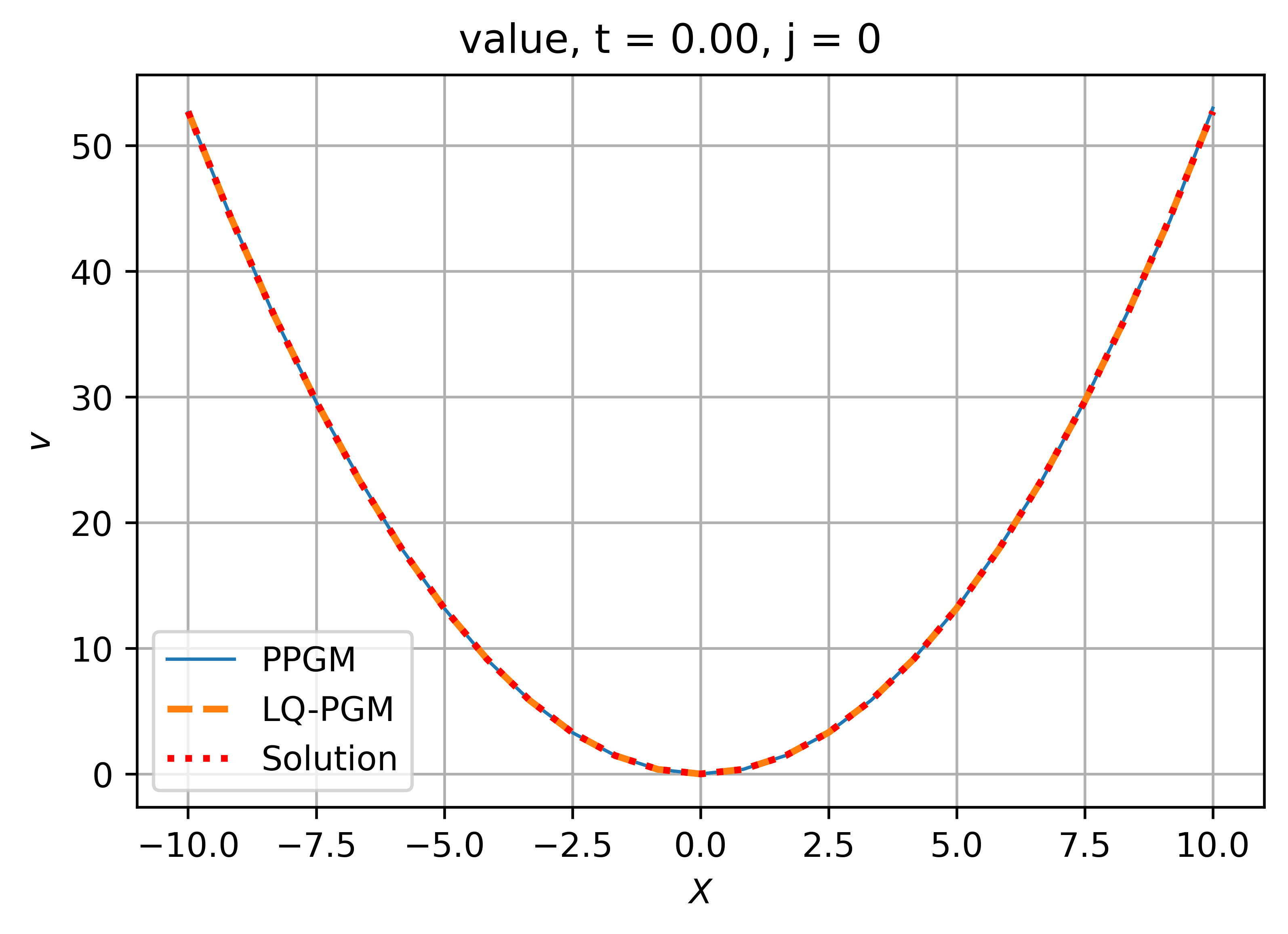}
\caption{Value function at time 0.}
\end{subfigure}
\end{minipage}
\caption{Comparison between algorithm errors for the unconstrained LQ problem. }
\label{fig_ex1_h2}
\end{figure}

\subsubsection{Convex problem with singular assumptions} \label{sec_singular}

Data used are $n = 2, m = 2$, $T = 1$,  
$Q=0$, $S=0$, $R=0$ and 
$$A  = \begin{pmatrix}
0.292 & 0.11 \\
-0.25 & 0.234 
\end{pmatrix},\
B  = \begin{pmatrix}
-0.177 & -0.204 \\
-0.157 & 0.077 
\end{pmatrix},\
C  = \begin{pmatrix}
0.052 & 0.019 \\
0.04 & 0.093 
\end{pmatrix},
$$
$$
D  = \begin{pmatrix}
0.852 & 0.189 \\
-0.236 & 1.085 
\end{pmatrix},\
G  = \begin{pmatrix}
1.376 & 0.01 \\
0.01 & 0.539 
\end{pmatrix}.
$$
This problem is convex  and satisfies the singular assumptions with $\mu = 0.539$, $\A$ and $\D - \B^\top \A^{-1}\B$ positive definite.  
In this example, the default algorithm configuration leads to stability issues. The number of sub-steps (100) for each sub-algorithm appears to be too high, and errors propagate and diverge. Instead, we reduce the number of sub-steps to $l_2 = l_3 = 10$ for this case, and the control and BSDE sub-algorithms use learning rates $\rho = 0.002$ and $\delta = 0.005$, respectively. We increase the batch size to $M = 100$, and run for $k_1 = 500$ total iteration steps.
We plot the same graphs for this example as the previous one, see Figure \ref{fig_ex2_h2}. The PPGM algorithm runs for a much higher number of steps than the LQ-PGM algorithms, but the individual steps are less computationally costly than for the standard example. The large control error is possibly due to the algorithm design changes, which achieves lower accuracy as a price of stability, but it appears the control error does not impact the value function, as it is still accurately found.

\begin{figure}[H]
\centering
\begin{minipage}{.33\textwidth}
\centering
\begin{subfigure}[b]{\textwidth}
\includegraphics[width=\textwidth]{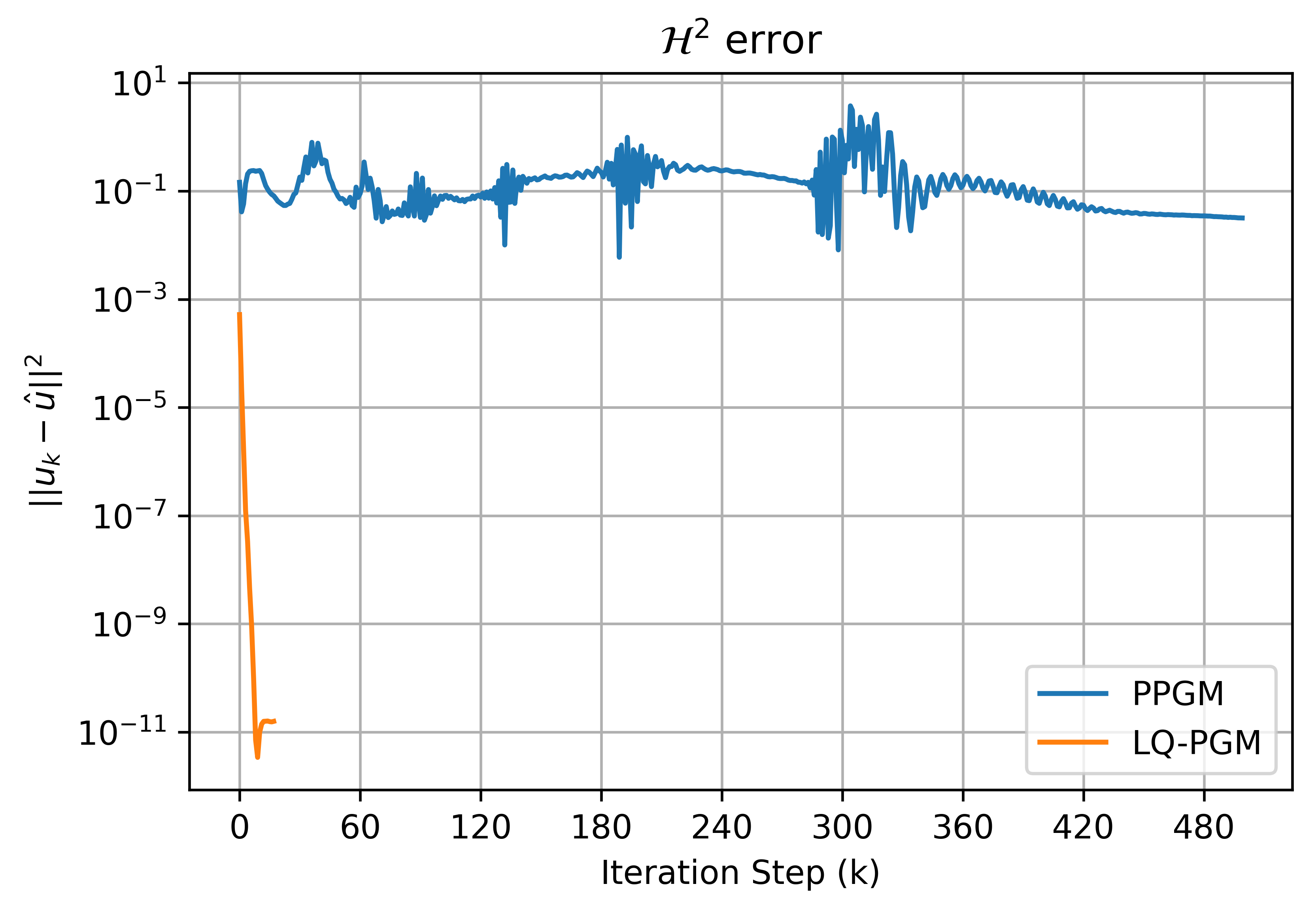}
\caption{Control error.}
\end{subfigure}
\end{minipage}%
\begin{minipage}{.33\textwidth}
\centering
\centering
\begin{subfigure}[b]{\textwidth}
\includegraphics[width=\textwidth]{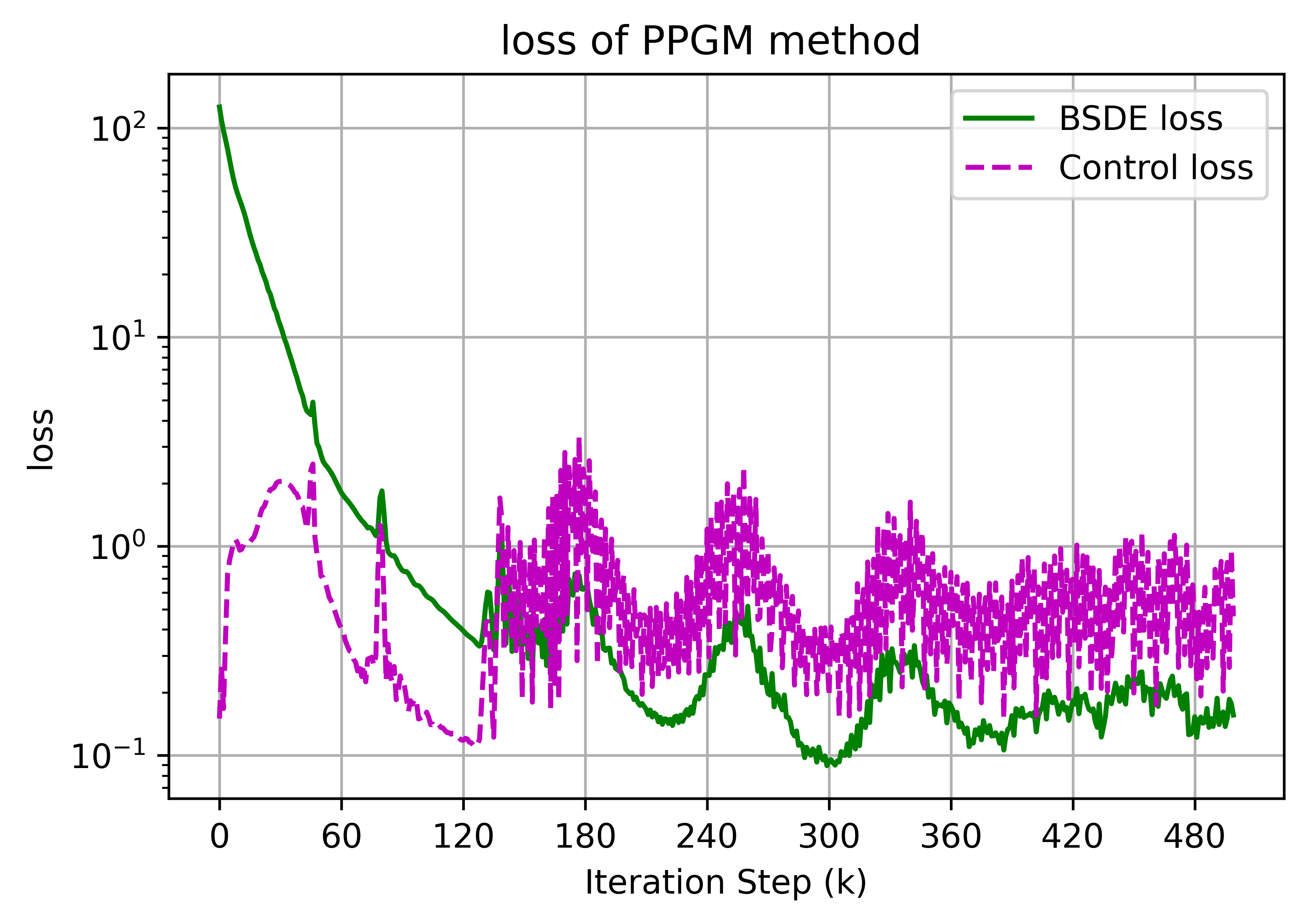}
\caption{Loss functions for  PPGM  }
\end{subfigure}
\end{minipage}
\begin{minipage}{.33\textwidth}
\centering
\centering
\begin{subfigure}[b]{\textwidth}
\includegraphics[width=\textwidth]{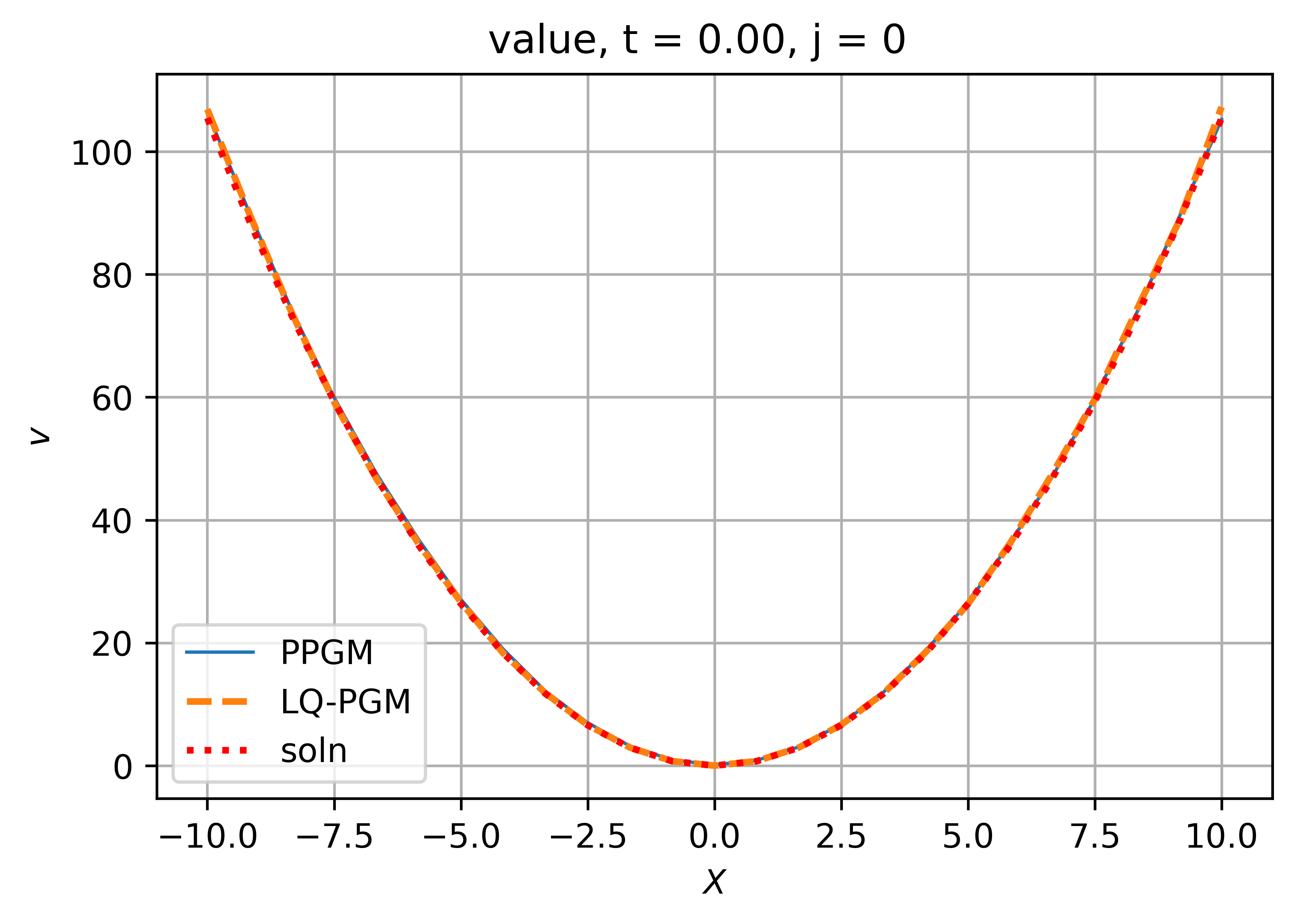}
\caption{Value function at time 0.}
\end{subfigure}
\end{minipage}
\caption{Comparison between algorithm errors for the unconstrained LQ problem. }
\label{fig_ex2_h2}
\end{figure}

\subsubsection{Non-Convex problem with standard assumptions}\label{nonconvex}

Data used are $n = 1, m = 5$, $T = 1$,  
$A  = 0.083$, $C  = -0.314$,   $G  = 0.78$ and
\begin{align*}
B & = \begin{pmatrix}
0.22 & -0.5 & -0.198 & -0.353 & -0.408 
\end{pmatrix}\\
D & = \begin{pmatrix}
-0.154 & -0.103 & 0.039 & 0.081 & 0.185 
\end{pmatrix}\\
\begin{pmatrix}
Q & S^\top \\
S & R
\end{pmatrix} 
& = \begin{pmatrix}
0.2 & 0.335 & -0.482 & 0.25 & 0.489 & 0.248 \\
0.335 & 0.868 & 0.098 & -0.147 & 0.065 & 0.084 \\
-0.482 & 0.098 & 0.839 & 0.025 & 0.007 & 0.112 \\
0.25 & -0.147 & 0.025 & 1.168 & 0.173 & 0.05 \\
0.489 & 0.065 & 0.007 & 0.173 & 0.82 & 0.059 \\
0.248 & 0.084 & 0.112 & 0.05 & 0.059 & 1.083 
\end{pmatrix}. 
\end{align*}

This problem is non-convex due to the $S$ term,  even though the running cost is strongly convex in control with $\mu = 0.73$.  The value function is concave, see
Figure \ref{fig_ex3_h2}. The algorithms successfully converge in this non-convex scenario.

\begin{figure}[H]
\centering
\begin{minipage}{.33\textwidth}
\centering
\begin{subfigure}[b]{\textwidth}
\includegraphics[width=\textwidth]{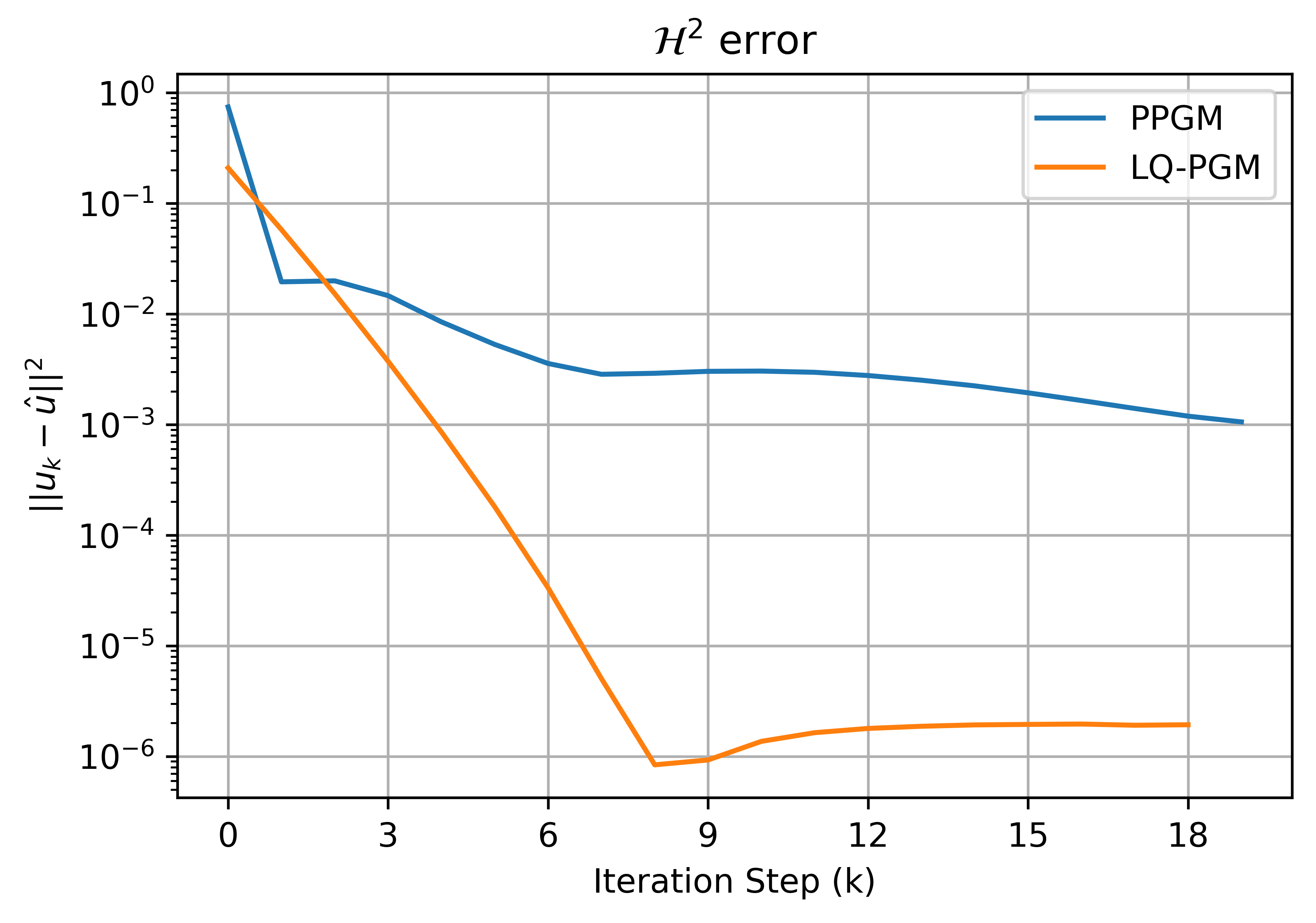}
\caption{Control error.}
\end{subfigure}
\end{minipage}%
\begin{minipage}{.33\textwidth}
\centering
\centering
\begin{subfigure}[b]{\textwidth}
\includegraphics[width=\textwidth]{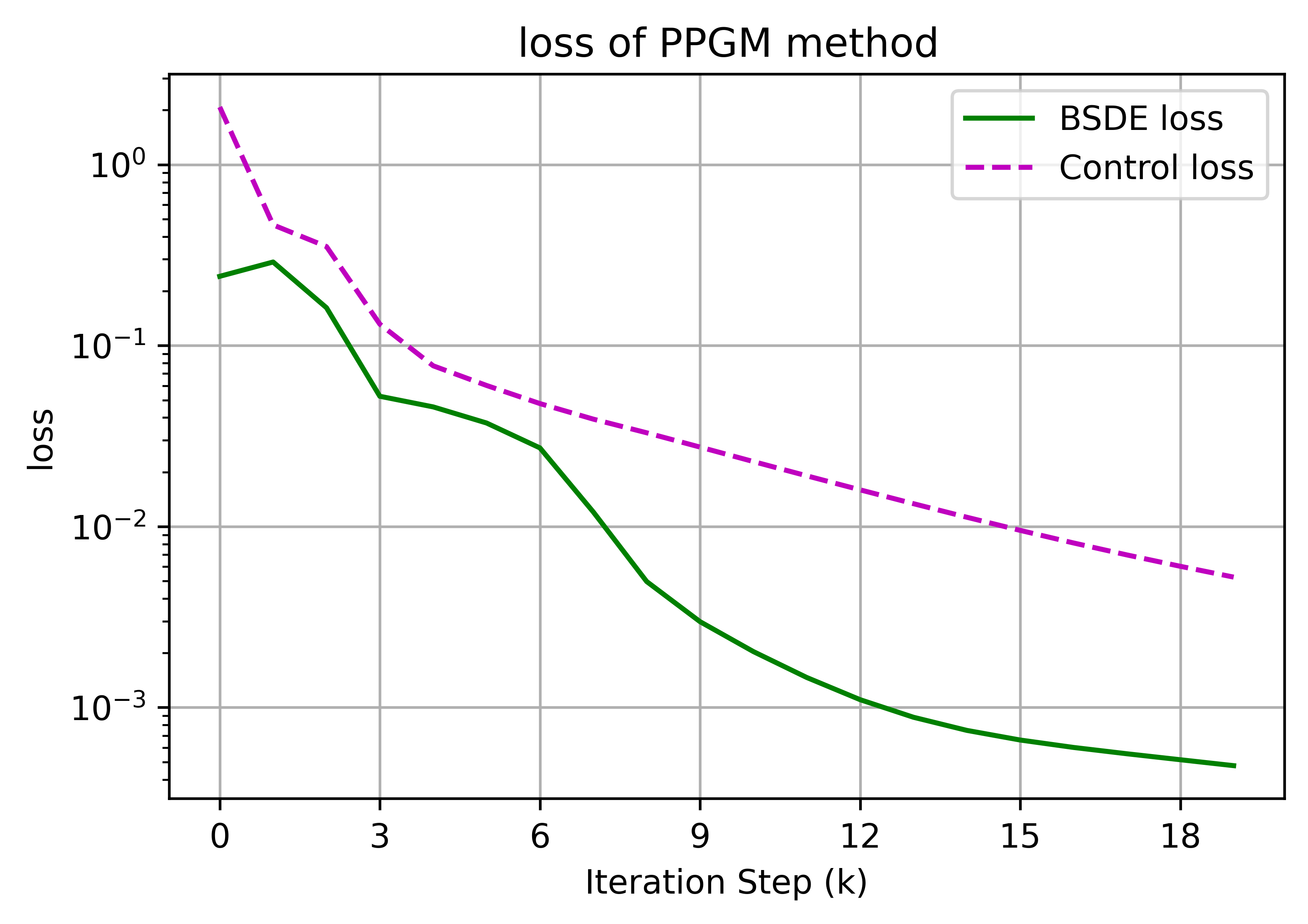}
\caption{Loss functions for  PPGM }
\end{subfigure}
\end{minipage}
\begin{minipage}{.33\textwidth}
\centering
\centering
\begin{subfigure}[b]{\textwidth}
\includegraphics[width=\textwidth]{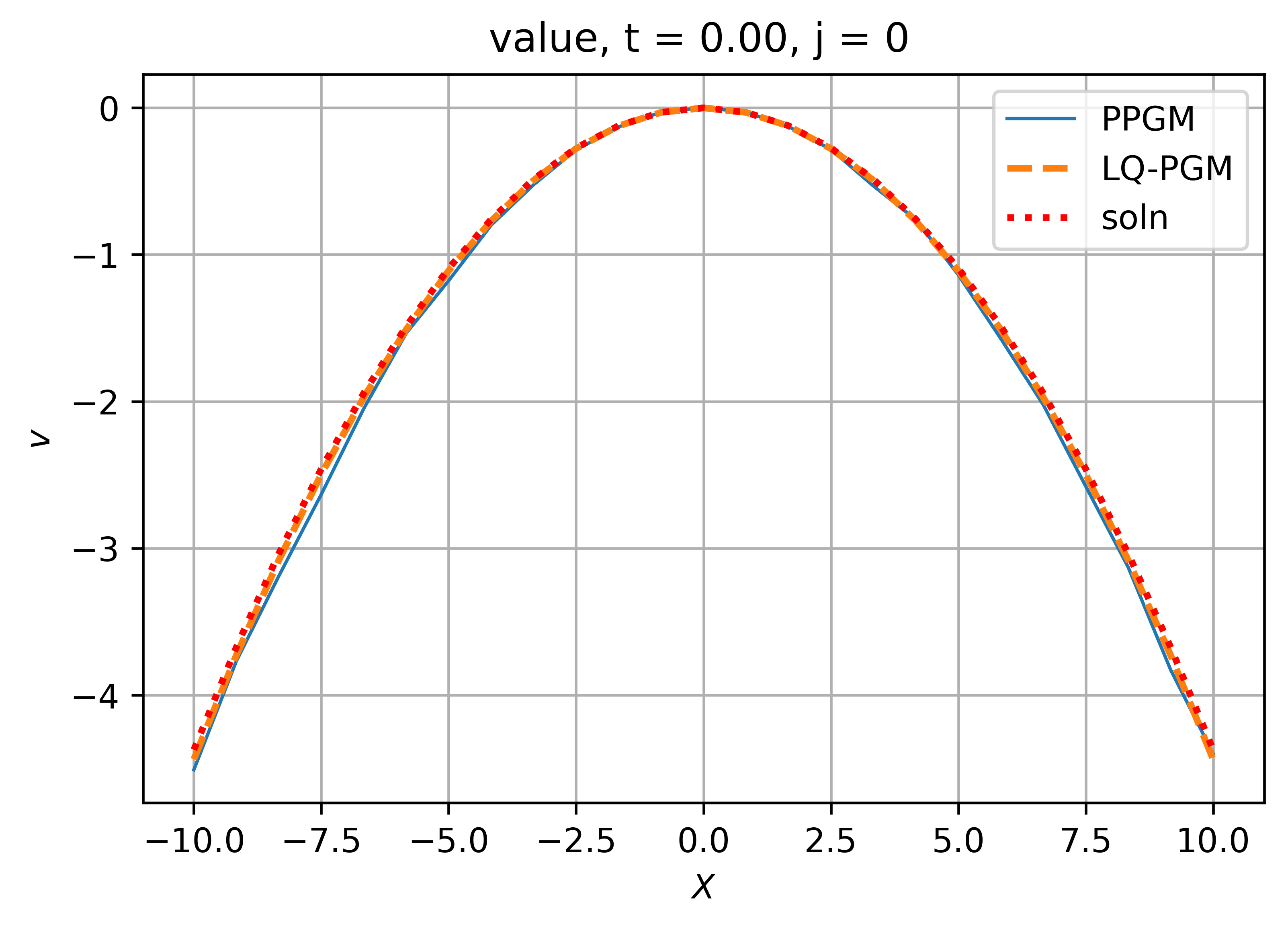}
\caption{Value function at time 0.}
\end{subfigure}
\end{minipage}
\caption{Comparison between algorithm errors for the unconstrained LQ problem. }
\label{fig_ex3_h2}
\end{figure}

\subsubsection{Cone constrained convex problem with standard assumptions}

Data used are the same as ones in Example \ref{nonconvex}, and control constraint set is 
$U = \R^5_+$, the positive cone. 
The HJB equation (\ref{eq_pde}) for this problem is given by
\begin{align*}
\frac{\partial v}{\partial t} + \inf_{u \in U} \left\{ \left(A_t x + B_t u\right)^\top \frac{\partial v}{\partial x} + \half \left(C_t x + D_t u\right)^\top \frac{\partial^2 v}{\partial x^2} \left(C_t x + D_t u \right) + \half x^\top Q_t x + \half u^\top R_t u \right\} = 0
\end{align*}
with terminal condition $v(T, x) = (1/2) x^\top G x$. Following \citep{hu2005constrained}, we propose that the optimal control and value can be written in the following form
\begin{align*}
u^*_t(x) & = \xi^+_t \max(0, x) + \xi^-_t \max(0, -x) \\
v_t(x) & =  P^+_t\max(0, x)^2 +  P^-_t \max(0, -x)^2,
\end{align*}
where $\xi^+_t  = \argmin_{\xi \in \R^{ m}_+} H^+(t, \xi, P^+_t)$, 
$\xi^-_t  = \argmin_{\xi \in \R^{ m}_+} H^-(t, \xi, P^-_t)$, and $H^+, H^-$ are defined by
\begin{align*}
H^\pm(t, \xi, P) & := \xi^\top \left(\half R_t + D_t^\top P D_t\right)\xi \pm 2  \xi^\top \left( B_t^\top P  + D_t^\top P C_t + S_t \right), 
\end{align*}
and $P^+$, $P^- \in \C^0([0,T]; \R)$ solve the ODEs
\begin{align*}
\frac{dP^\pm}{dt} & = -\left(\left(A_t + A_t^\top + C_t^\top C\right)P^\pm + \half Q + \min_{\xi \in \R^{ m}_+} H^\pm(t, \xi, P^\pm_t) \right), 
\end{align*}
with $P^+_T = P^-_T = (1/2)G$. 
 Figure \ref{fig_ex4_h2} shows that, unlike Example \ref{nonconvex}, the control constraint ensures that the value function is convex, in contrast to the value function of  Example \ref{nonconvex} (labelled "LQ-PGM") that has a lower value but infeasible. 

\begin{figure}[H]
\centering
\begin{minipage}{.33\textwidth}
\centering
\begin{subfigure}[b]{\textwidth}
\includegraphics[width=\textwidth]{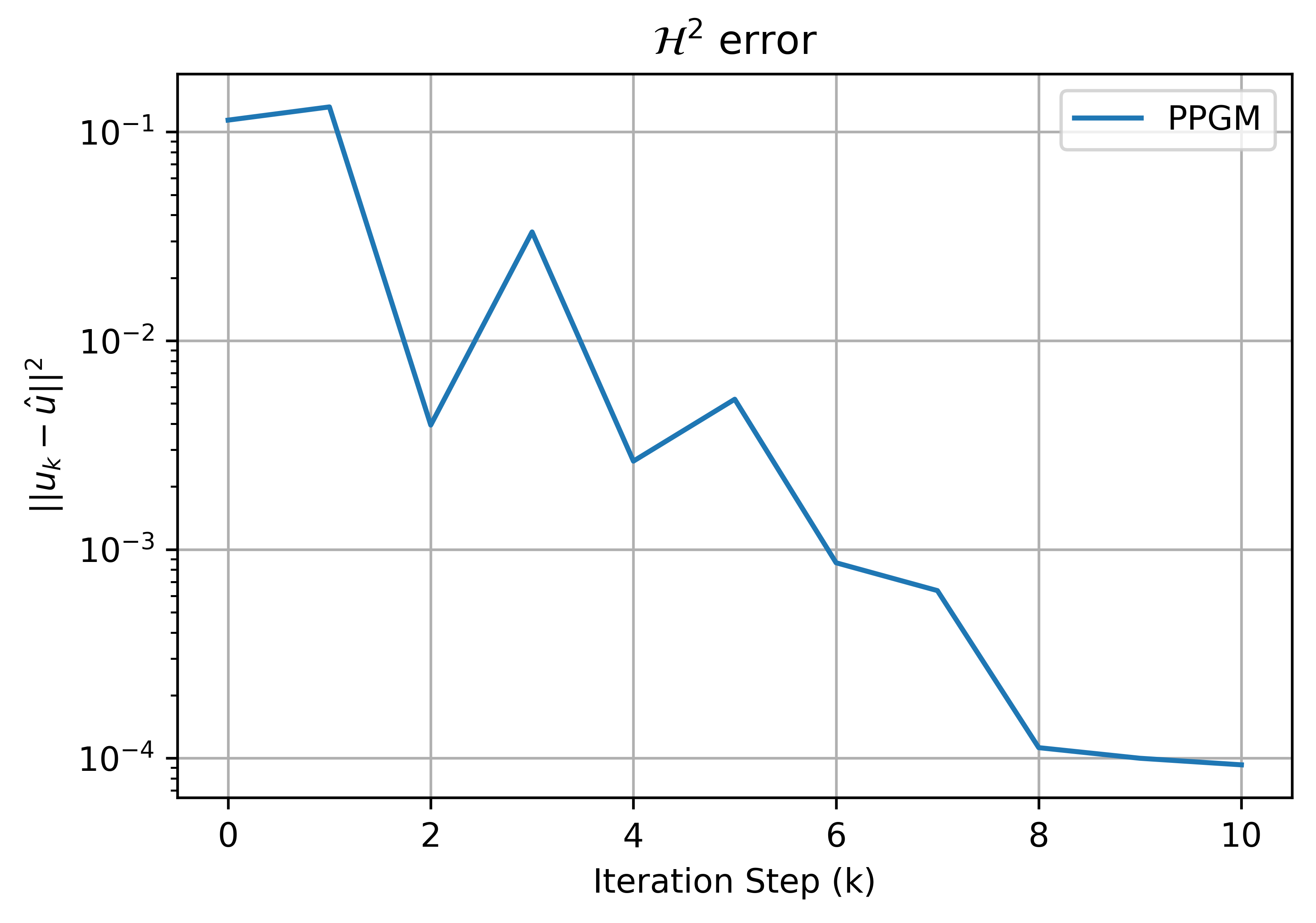}
\caption{Control error.}
\end{subfigure}
\end{minipage}%
\begin{minipage}{.33\textwidth}
\centering
\centering
\begin{subfigure}[b]{\textwidth}
\includegraphics[width=\textwidth]{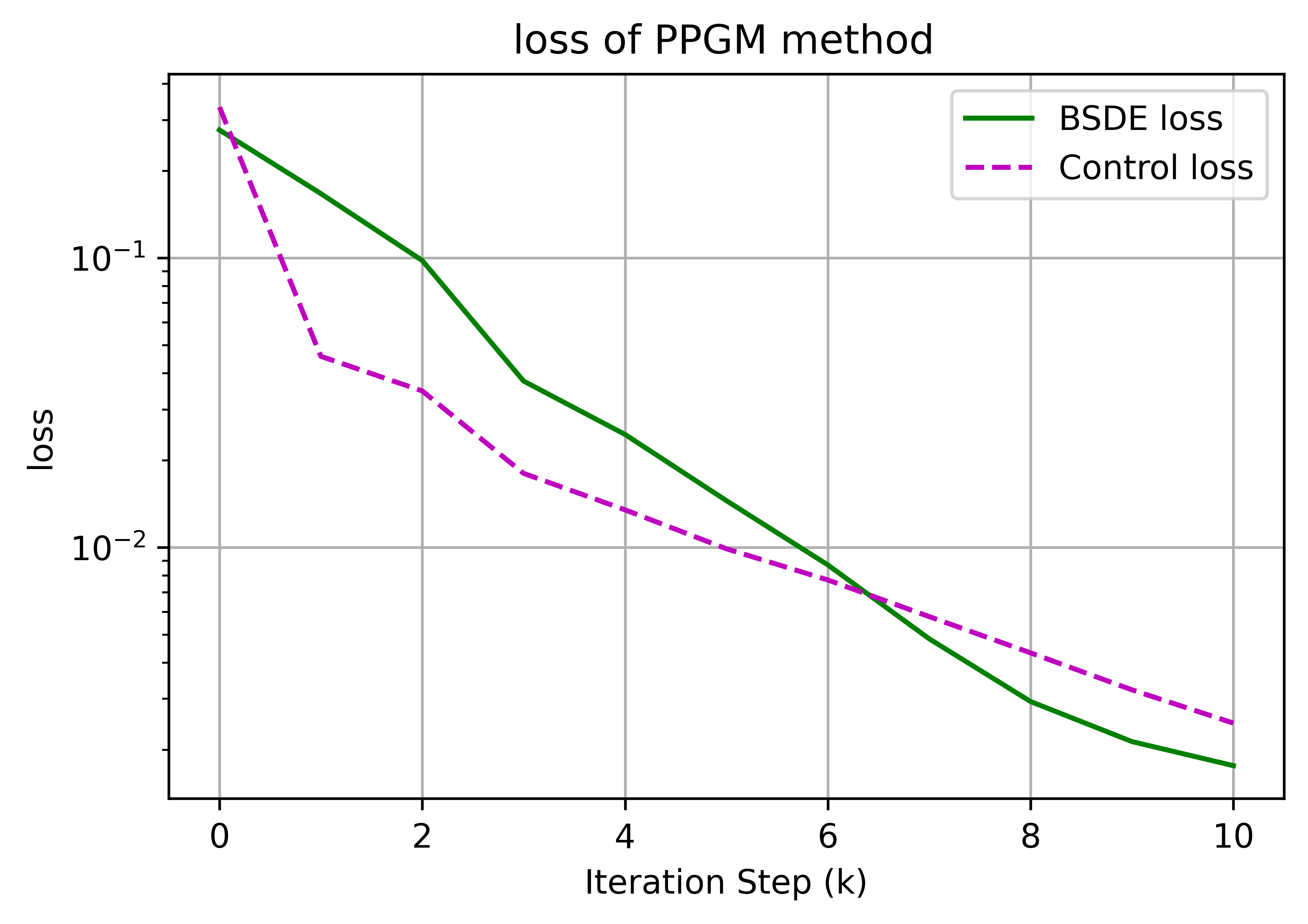}
\caption{Loss functions for  PPGM  }
\end{subfigure}
\end{minipage}
\begin{minipage}{.33\textwidth}
\centering
\centering
\begin{subfigure}[b]{\textwidth}
\includegraphics[width=\textwidth]{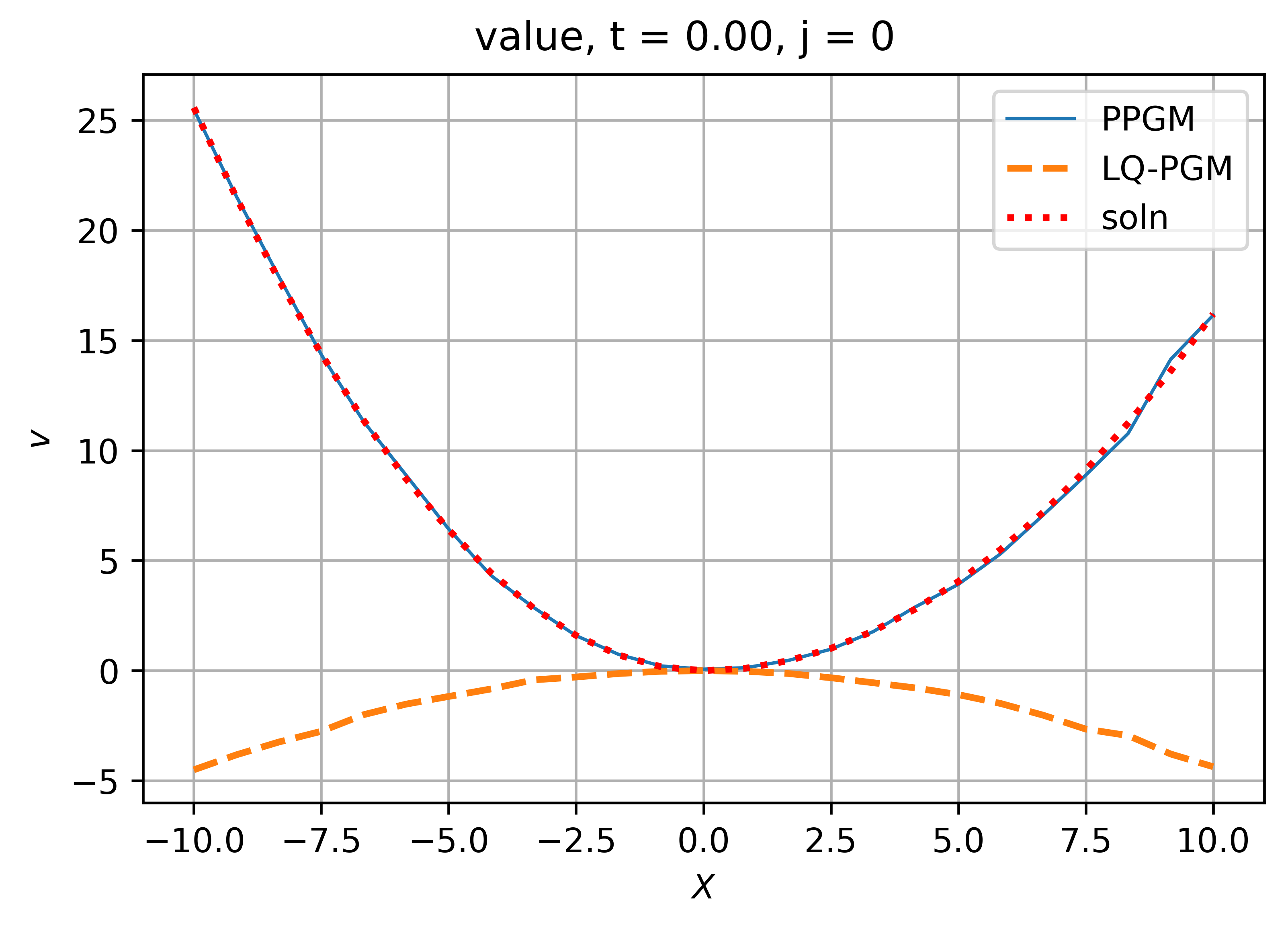}
\caption{Value function at time 0.}
\end{subfigure}
\end{minipage}
\caption{Comparison between algorithm errors for the constrained LQ problem. }
\label{fig_ex4_h2}
\end{figure}

\subsubsection{Sensitivity Analysis}

In this section we consider sensitivity and robustness of the LQ-PGM algorithm. We choose $n = m$ and take coefficients dependent on this dimension to scale the problem. We initialise $A_t, B_t, C_t, D_t \sim \Unif(-0.25 n^{-1}, 0.25 n^{-1})$ as constant matrices, scaled by $n^{-1}$. This scaling factor is used to ensure that the set of coefficients satisfies the assumptions of the convergence analysis for each $n$, and was chosen by inspection. We set $G = \ind_{n} + \half(G_0 + G_0^\top)$ where $G_0 \sim \Unif(-0.5, 0.5)$, to ensure that $G$ is symmetric. We set $Q_t$ to be a matrix with each entry equal to 0.2, $S_t = S \sim \Unif(-0.5, 0.5)$, and finally $R_t = \ind_n + R$ where $R \sim \text{Unif}\left(-0.5n^{-1}, 0.5n^{-1}\right) $ matrix.  This scaling ensures that the standard assumptions are satisfied for the problem for any dimension choice $n$ and random matrix $R$. We set $T = 1.0$ throughout. 

\begin{example} \label{ex_uncon}  We test if the theoretical exponential convergence rate is matched by the algorithm with $n = m = 100$. The algorithm runs for 20 steps in about 45 seconds, when the criteria value $\Delta_k$ drops below $10^{-6}$.
Figure \ref{fig_pgm} (a) numerically verifies the first statement of Theorem \ref{thm_convergence}. We run the algorithm 10 times with 10 different randomly generated model coefficients, and compare to the analytical solutions. We plot the mean and standard deviation of the relative control and value errors, given at iteration step $k \in \N$ by
$\|\alpha^k - \alpha^*\|_\infty/\|\alpha^*\|_\infty$ and $\|a^k - a^*\|_\infty/\|a^*\|_\infty$ respectively. 
We see that the log-scaled decay is linear, agreeing with the results of Theorem \ref{thm_convergence} (i), until the convergence criteria is reached. The error flattening out indicates that the algorithm cannot produce more accuracy beyond standard numerical errors.


\end{example}

\begin{example}

We now consider changes in the dimensions $n$ and $m$ of the problem. As the underlying coefficients depend on $n$ (in order to ensure the assumptions of the convergence results hold), the accuracy of the algorithm artificially increases as $n$ increases. We therefore only consider the effect on the runtime. Figure \ref{fig_pgm} (b) shows the effect of the runtime on the dimension when we set $n = m$ to range between 1 and 100. We run 5 times and take an average. The log-scaled graph is sub-linear in dimension, suggesting that the time-dependence is not exponential. 

\end{example}

\begin{example}
We now consider the impact of strong convexity on the problem, under the standard conditions. Let all coefficients other than $R$ be fixed and define 
$R_t = r\ind_{m,m}$, for all $t \in [0,T]$,
for some $r > 0$. In this instance we have $\mu = \left\|R\right\|_\infty = r $.
We see that large values of $r$ ensure the assumptions of Theorem \ref{thm_convergence} are satisfied. We fix $n = m = 5$ and consider the problem while varying only the value of $r$. Figure \ref{fig_pgm} (c) shows the final error of the algorithm for different values of $r$. As $r$ decreases the error increases dramatically, appearing to increase exponentially even with log scaling.
We see that when $r$ is smaller than some critical value $r_{\min}$, the algorithm does not accurately return the value at time 0 and appears to be chaotic. For this example, it appears that the critical value is approximately $r_{\min} \approx 0.1$.
\end{example}

\begin{figure}[H]
\centering
\begin{minipage}{.33\textwidth}
\centering
\begin{subfigure}[b]{\textwidth}
\includegraphics[width=\textwidth]{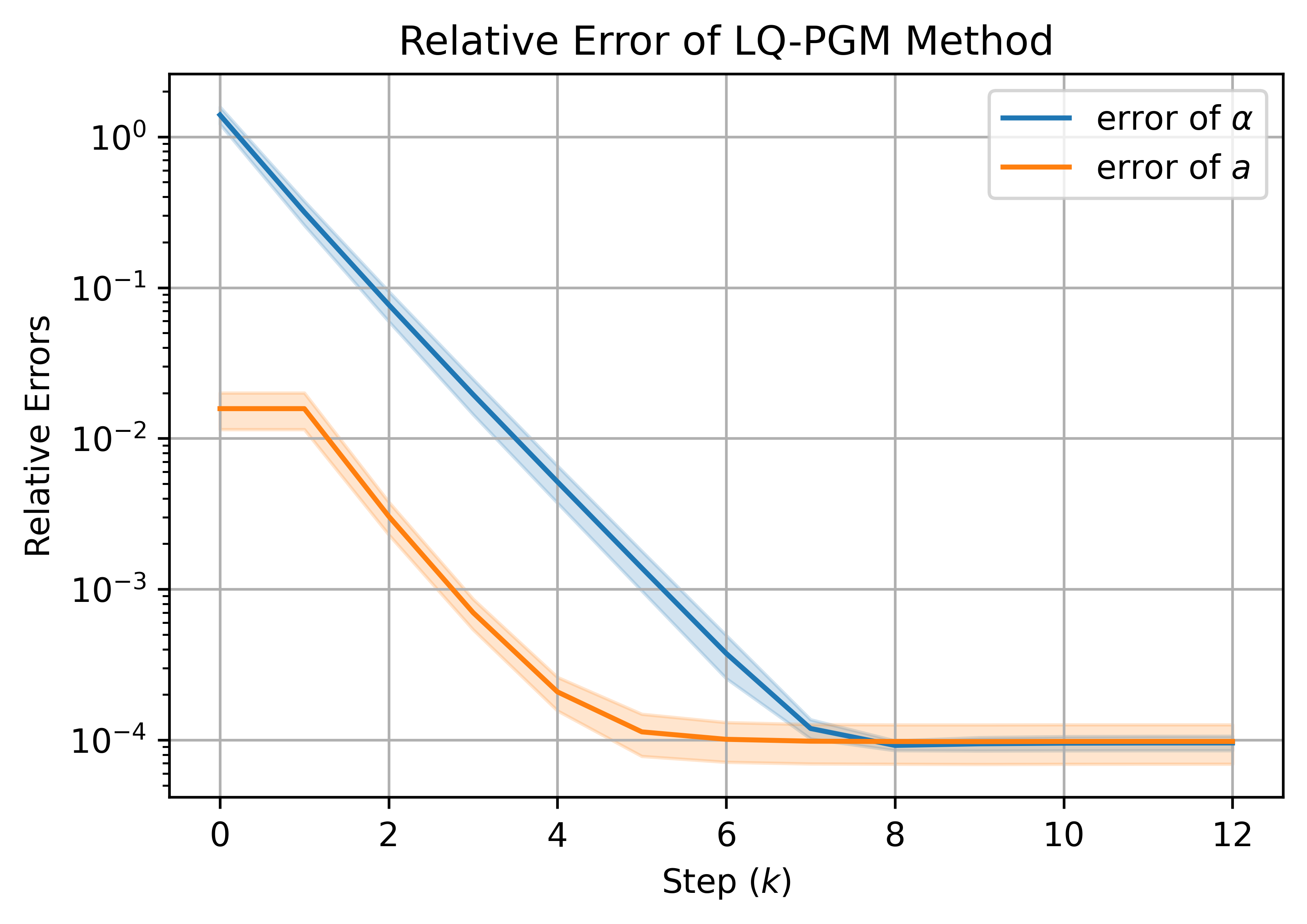}
\caption{Control error (log-scaled).}
\end{subfigure}
\end{minipage}%
\begin{minipage}{.33\textwidth}
\centering
\centering
\begin{subfigure}[b]{\textwidth}
\includegraphics[width=\textwidth]{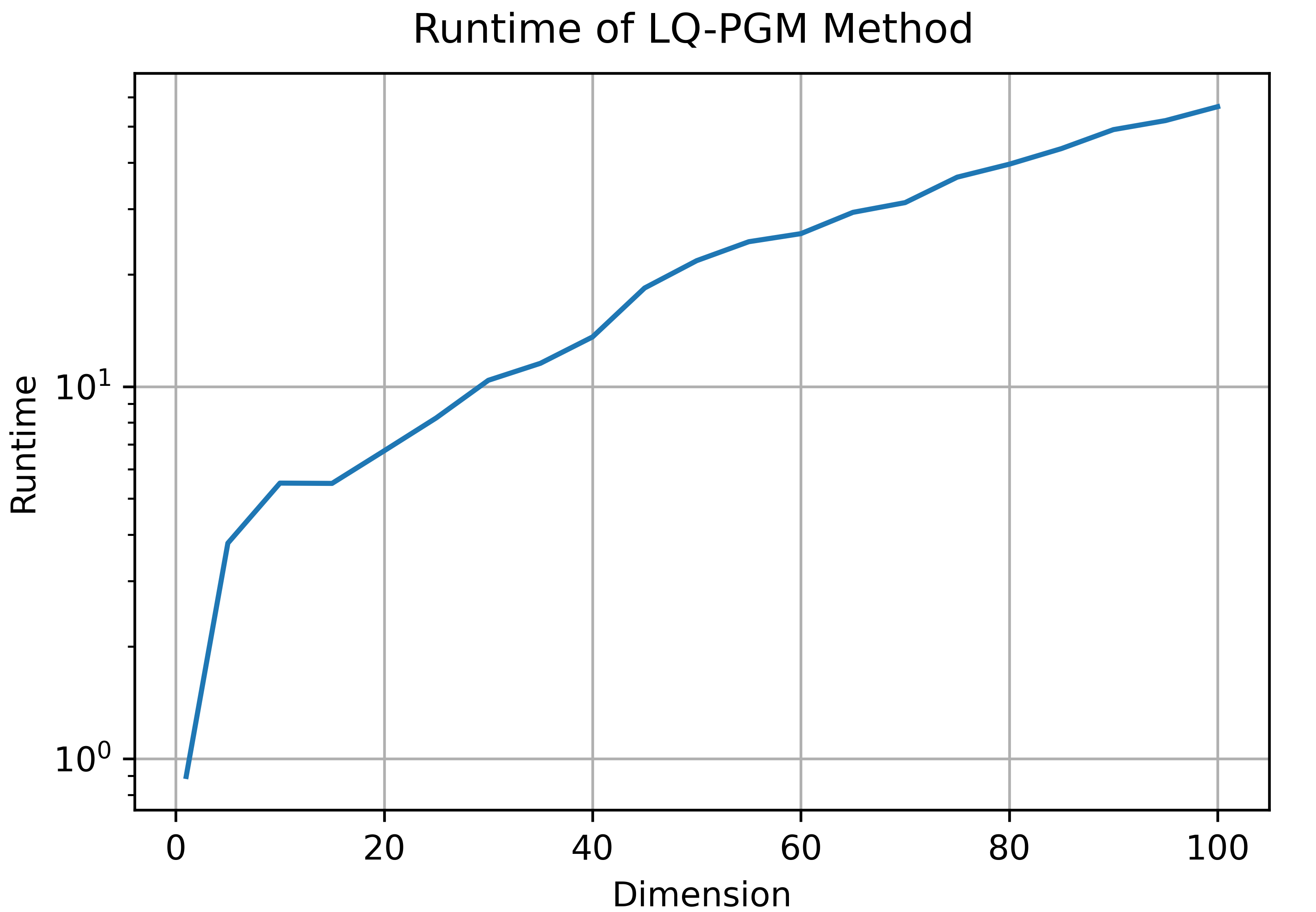}
\caption{Runtime (log-scaled)}
\end{subfigure}
\end{minipage}
\begin{minipage}{.33\textwidth}
\centering
\centering
\begin{subfigure}[b]{\textwidth}
\includegraphics[width=\textwidth]{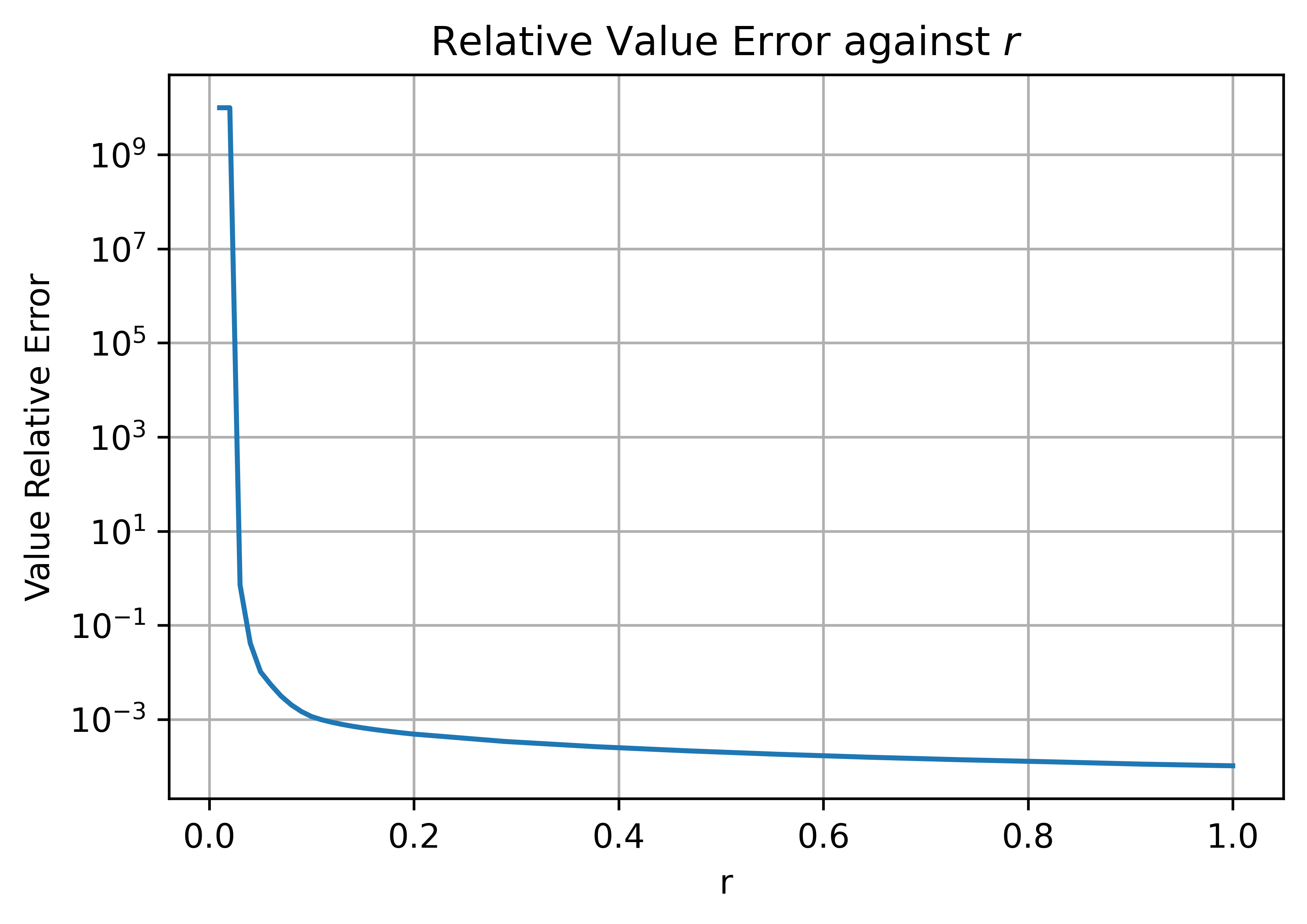}
\caption{Value error vs $r$ (log-scaled).}
\end{subfigure}
\end{minipage}
\caption{Control and value errors and runtime of unconstrained PGM applied to LQ problem.}
\label{fig_pgm}
\end{figure}
\subsection{Convex Problem with Non-Polynomial Running Cost}

We consider the following unconstrained problem with non-convex  running cost
\begin{align*}
\inf_{u \in \U} \E& \left[ - \int_0^T \sum_{i = 1}^m \cos(u_s^i)
ds  + \half \left|X^u_T\right|^2 \right],
\end{align*}
where $X^u$ satisfies SDE (\ref{eq_state}), $m=n$, $D$ is a nonsingular matrix,  $A= - (1/2) B D^{-1}D^{-1}$ and $C=(D^{-1})^\top B^\top$. By construction we have $\A = \B = 0$, and the singular case of Assumption \ref{ass_strong} is satisfied. By Ito's formula, we have
\begin{align*}
\E\left[\left|X^u_T\right|^2\right] & = x_0^2 + \left[\int_0^T \left|D_t u_t\right|^2 dt \right].
\end{align*}
It is clear that $u = 0$ is an optimal control for this problem.
Take $n = 3, m = 3$, $T = 1.0$, and
$$
B  = \begin{pmatrix}
0.013 & 0.027 & 0.062 \\
0.099 & 0.126 & -0.158 \\
0.057 & 0.028 & 0.02 
\end{pmatrix}, \
D  = \begin{pmatrix}
1.346 & -0.11 & 0.126 \\
-0.134 & 1.474 & 0.153 \\
0.011 & 0.064 & 1.439 
\end{pmatrix},
$$
where $B$ and $D$ are generated randomly, with 1.5 added to the diagonal elements of $D$. In this case, we have $\A = \B = 0$ and  $\D = \D - \B^\top \B$ positive definite.
Figure \ref{fig_ex5_h2} shows the convergence of the general PPGM algorithm to the optimal control $u^* = 0$. 

\begin{figure}[H]
\centering
\begin{minipage}{.33\textwidth}
\centering
\begin{subfigure}[b]{\textwidth}
\includegraphics[width=\textwidth]{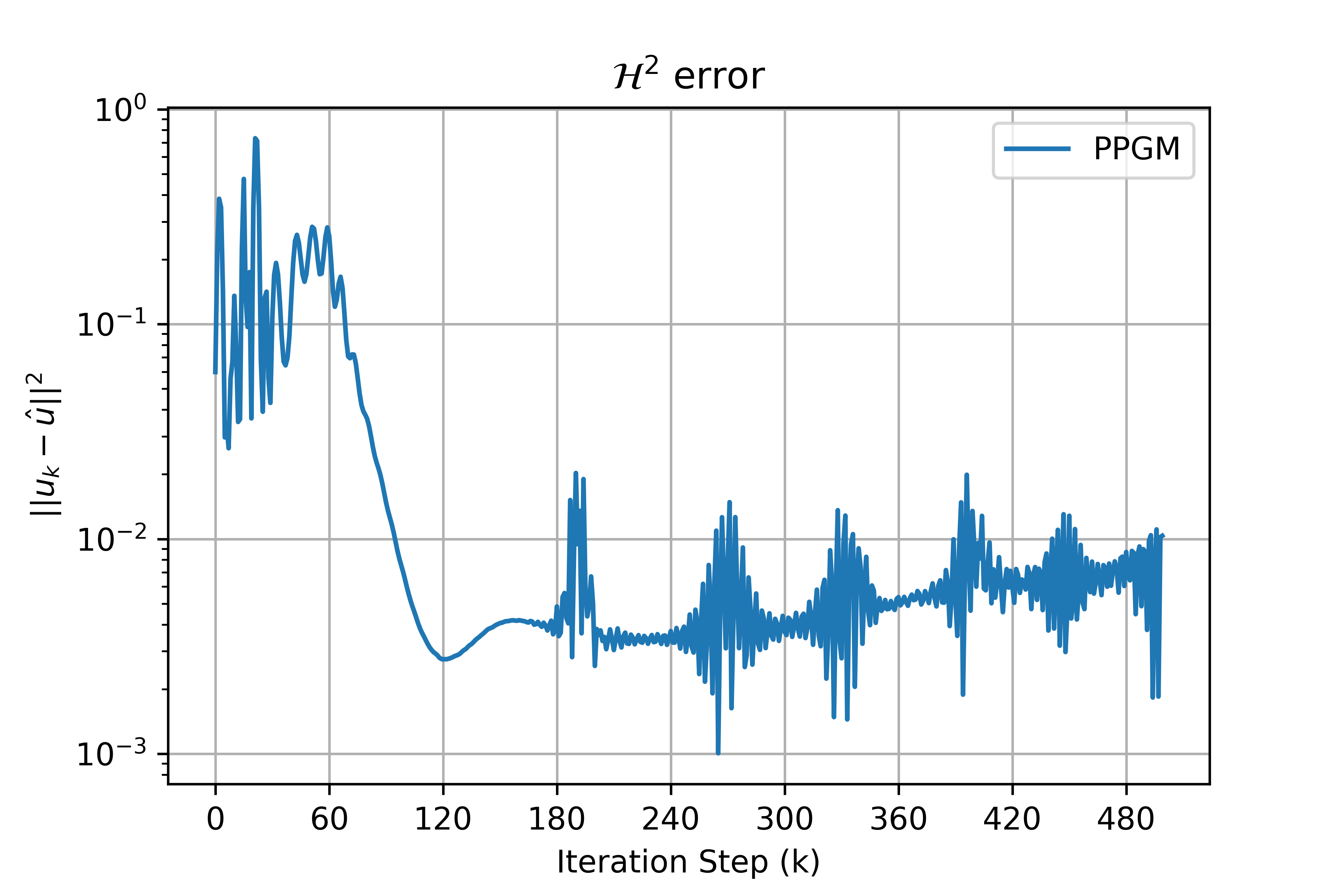}
\caption{Control error.}
\end{subfigure}
\end{minipage}%
\begin{minipage}{.33\textwidth}
\centering
\centering
\begin{subfigure}[b]{\textwidth}
\includegraphics[width=\textwidth]{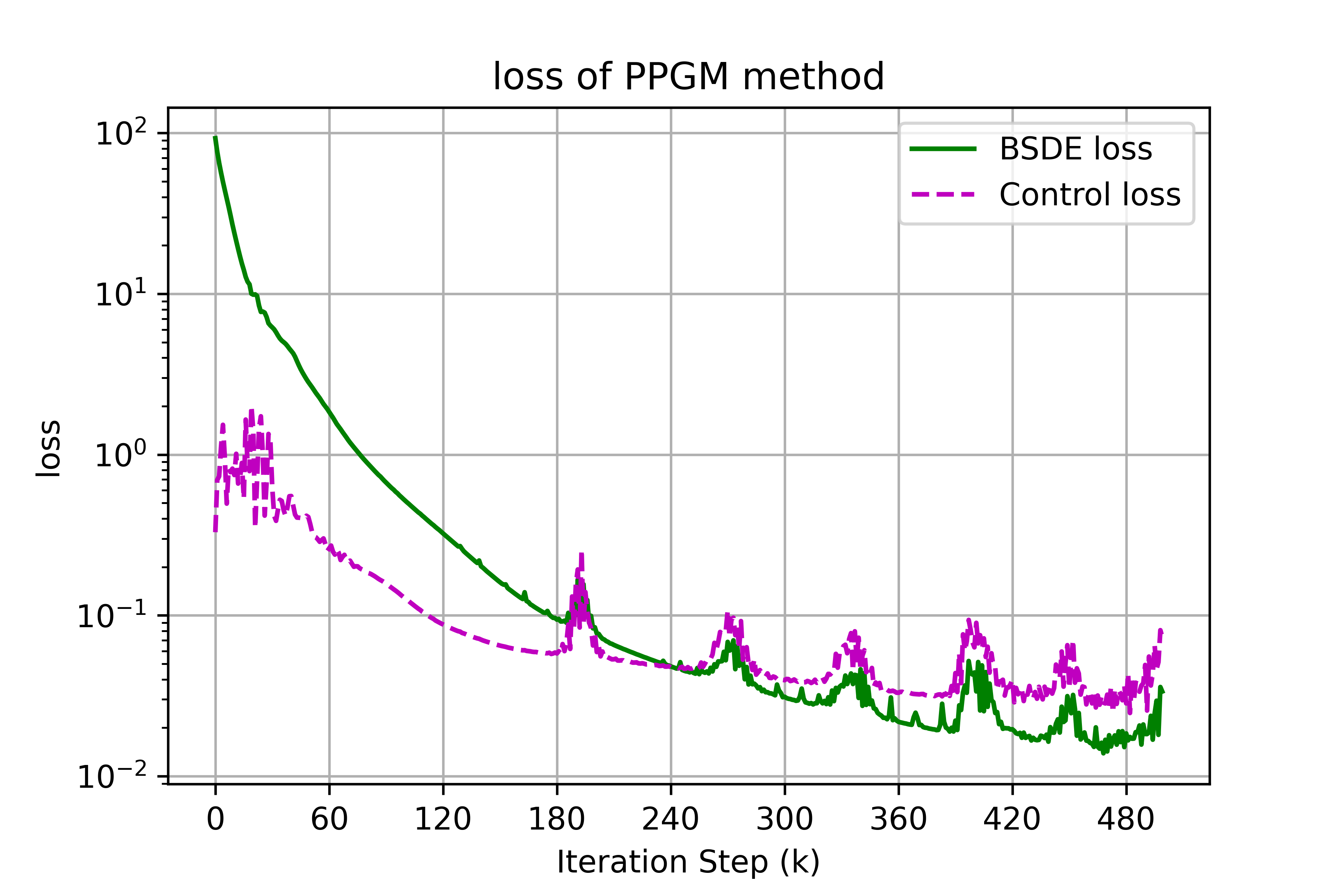}
\caption{Loss functions for  PPGM  }
\end{subfigure}
\end{minipage}
\begin{minipage}{.33\textwidth}
\centering
\centering
\begin{subfigure}[b]{\textwidth}
\includegraphics[width=\textwidth]{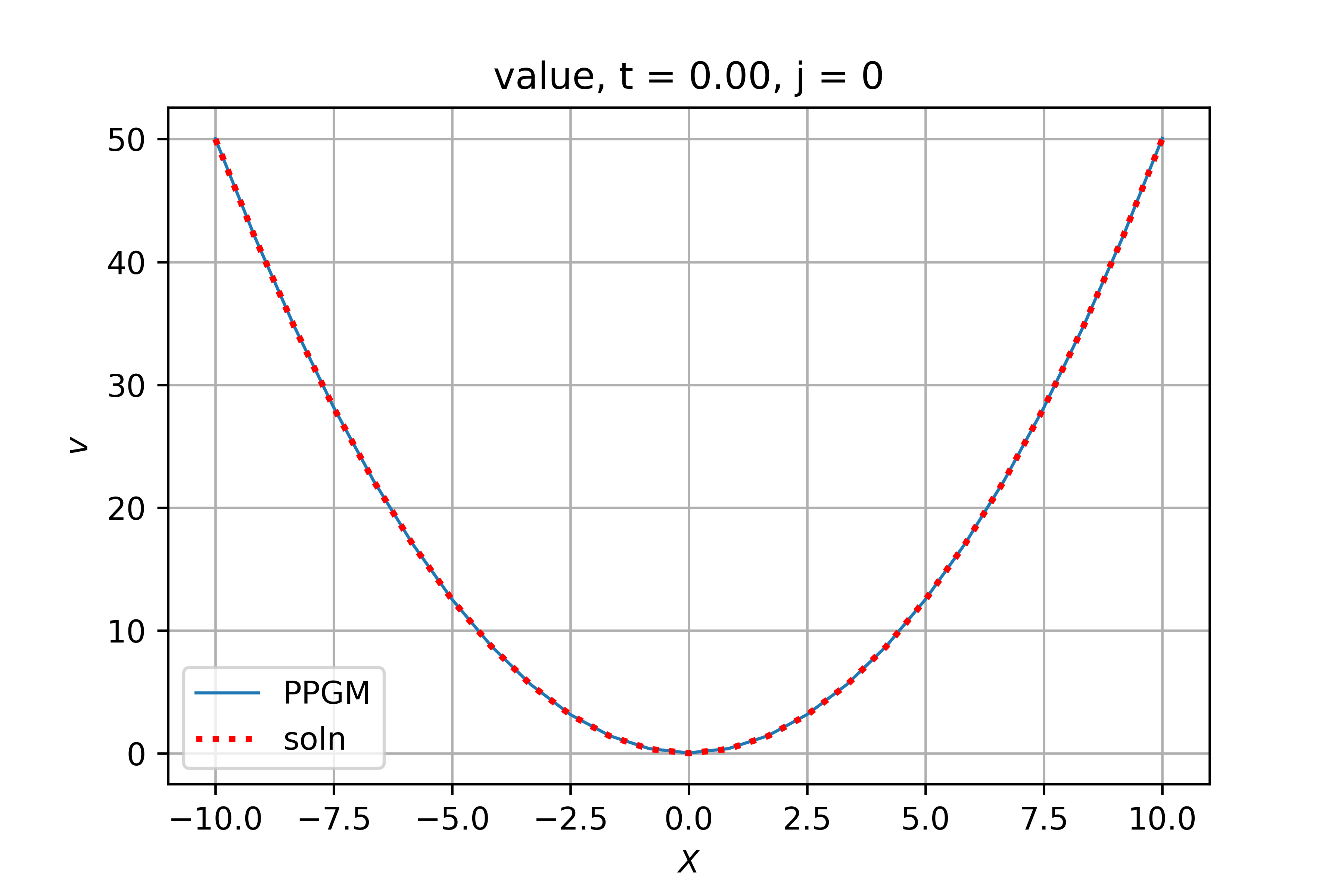}
\caption{Value function at time 0.}
\end{subfigure}
\end{minipage}
\caption{Algorithm results for non-strongly-convex running cost. }
\label{fig_ex5_h2}
\end{figure}

\section{Conclusion} \label{sec_conc}

In this paper we determine sufficient conditions on  the convergence of PPGM for stochastic control problems with control dependent diffusion coefficients. Our approach is novel in the literature in the sense that we do not directly discuss the regularity of  the solutions $(Y,Z)$  to the adjoint BSDE, which is difficult, especially for $Z$ process, we  instead give an alternative representation of the gradient of the Hamiltonian in terms of the adjoint operators induced by the linear dynamics of the state process. We then show the linear convergence of the PPGM under strong convexity conditions on either running or terminal cost functions, the latter requiring further strong convexity of the state process with respect to the control. We introduce the deep neural network algorithms to implement our PPGM and present numerical results that capture the theoretical convergence results.
Our convergence analysis crucially relies on linear state process. If drift and diffusion coefficients are nonlinear functions, then we may only express $Z^k$ in terms of Malliavin derivatives, which would lead to great complexity as regularity estimates of $Z^k$ require even stricter regularity estimates on the PPGM control $\phi^k$, a function of the process $Z^{k-1}$ from the previous step, we leave this and other questions to future research.  



\appendix
\section{Proofs} 

We show that the sequence $(u^k)_{k \in \N}$, starting at some $u^0 \in \U$ and determined by \eqref{update_control}, is contracting and converges to some element $u^* \in \HH^2(\R^m)$ in  $\left\|\cdot\right\|_{\HH^2}$ norm. We then show that the limiting control $u^{*}$ is a stationary point of the control problem. We denote by $K$ a generic constant that does not depend on $\mu$ or any control $u^k$, and may change from line to line.

%

Using the regularity for the FBSDE solutions, we first show the update formula (\ref{update_control}) is well defined. 

\begin{lemma}\label{lemma1}
Let $u^0 \in \U$. Then the sequence $(u^k)_{k \in \N}$ generated by \eqref{update_control} satisfies $u^k \in \U$ for all $k \in \N$. 
\end{lemma}

\begin{proof}
We prove by induction. Assume $u^k \in \U$ for some $k \in \N$, and let $(X^k, Y^k, Z^k)$ be the corresponding FBSDE solutions of \eqref{update_control}. Then $X^k, Y^k, Z^k \in \s^2\left(\R^n\right) \times \s^2\left(\R^n\right)\times \HH^2\left(\R^n\right)$ by \citep[Theorems 3.2.2 and 4.2.1]{zhang2017backward}.
By construction we have $u^{k+1}_t \in U$ for all $t \in [0,T]$, so it remains to show that $\left\|u^{k+1}\right\|_{\HH^2} < \infty$. As the function $\prox_{U}$ is 1-Lipschitz and, by definition of $H$ and the Lipschitz regularity of the cost function derivative, $\partial_u H$ is Lipschitz, hence
\begin{align*}
\left\|u^{k+1}\right\|_{\HH^2} & \leq \left\|\prox_{U}\left(u^k - \tau \partial_u H\left(X^k, u^k, Y^k, Z^k\right)\right) - \prox_{U}(0)\right\|_{\HH^2} + \|\prox_{U}(0)\|_{\HH^2} \\
& \leq \left\|u^k - \tau \partial_u H\left(X^k, u^k, Y^k, Z^k\right)\right\|_{\HH^2} + K \\
& \leq \left\|u^k\right\|_{\HH^2} + \tau \left\|\partial_u H\left(X^k, u^k, Y^k, Z^k\right) - \partial_u H\left(0, 0, 0, 0\right)\right\|_{\HH^2} + \tau \left\|\partial_u H\left(0, 0, 0, 0\right)\right\|_{\HH^2} + K \\
& \leq  K \left( \left\|Y^k\right\|_{\HH^2} + \left\|Z^k\right\|_{\HH^2} + \left\|X^k\right\|_{\HH^2} + \left\|u^k\right\|_{\HH^2} + 1 \right) + \tau \left\|\partial_u f\left(0, 0\right)\right\|_{\HH^2} \\
& < \infty
\end{align*}
as required, where $0$ denotes the zero process in the appropriate dimension, and $\partial_u H\left(X^k, u^k, Y^k, Z^k\right)$ denotes the process $\left(\partial_u H_t\left(X^k_t, u^k_t, Y^k_t, Z^k_t\right)\right)_{t \in [0,T]}$.
\end{proof}

Now we represent the solutions of the BSDE via linear adjoint operators. Following the approach of \citep[Section 2.2]{sun2016open}, we define a matrix valued process $\Phi \in \s^2(\R^{n \times n})$ 
satisfying 
\begin{align} \label{eq_def_phi}
d \Phi_s  = A_s \Phi_s ds + C_s \Phi_s dW_s, \ s \in [0,T],
\end{align}
with  $\Phi_0=I_n$, the identity matrix of dimension $n$.
Let $\mathbb{L}(E, F)$ denote the set of bounded linear maps between Hilbert spaces $E$ and $F$. 

For $t \in [0,T]$ define an operator $L_t \in \mathbb{L}\left(\HH^2([t, T]; \R^m), \HH^2([t, T]; \R^n)\right)$ by, for $s \in [t,T]$,
\begin{align}  
(L_t u)_s & = \Phi_s\int_t^s \Phi_r^{-1} \left(B_r - C_r D_r\right)u_r dr + \Phi_s \int_t^s \Phi_r^{-1}D_r u_r dW_r,  \label{eq_def_L} 
\end{align}
where here and in the sequel, we use $(L_t u)_s$ to denote the process $L_t u$ evaluated at time $s$. We furthermore define  $L_{t, T} \in \mathbb{L}(\HH^2([t, T]; \R^m), \LL^2(\R^n, \F_T))$ by
\begin{align} \label{eq_def_LT}
L_{t, T}u = (L_tu)_T.
\end{align}

Recall that for any linear operator $\Gamma \in \mathbb{L}(E, F)$ we can, with some help from the Reisz Representation Theorem, define the adjoint operator $\Gamma^* \in \mathbb{L}(F, E)$ such that 
$
\langle \Gamma x, y \rangle_F = \langle x, \Gamma^* y\rangle_E,
$ for all $x \in E$ and $y \in F$. 
For any $t \in [0,T]$ let $L^*_t$ and $L^*_{t, T}$ denote the adjoint operators of $L_t$ and $L_{t, T}$, given in (\ref{eq_def_L}) and (\ref{eq_def_LT}) respectively. As all model coefficients are bounded, these maps are all linear and bounded, with the same operator norm as their respective original operators, which we denote by $\|L_t\|$ and $\|L_{t, T}\|$ respectively. We use these adjoint operators to determine a representation for the BSDE solutions. This allows us to determine sufficient regularity estimates involving $Z^{k}_t$. 


\begin{lemma} \label{lem_rep} Let $(u^k)_{k \in \N}$ be the sequence generated by \eqref{update_control}, for any $u_0 \in \U$, with corresponding FBSDE solutions $(X^k, Y^k, Z^k)$. Then for all $k \in \N$, $t \in [0,T]$ we have
\[ \partial_u H_t\left(X^k_t, u^k_t, Y^k_t, Z^k_t\right)  = \left(L_{0, T}^* \triangledown g\left(X^k_T\right)\right)_t + \left(L_0^* \partial_x f\left(X^k, u^k\right)\right)_t + \partial_u f_t\left(X^k_t, u^k_t\right),\]
where $\partial_x f\left(X^k, u^k\right)$ denotes the process $\left(\partial_x f_s\left(X^k_s, u^k_s\right)\right)_{s \in [0,T]}$.
\end{lemma}

\begin{proof}
By (\ref{eq_def_H}), we have
$ \partial_u H_t\left(X^k_t, u^k_t, Y^k_t, Z^k_t\right)  = B_t^\top Y^k_t + D_t^\top Z^k_t + \partial_u f_t\left(X^k_t, u^k_t\right).$
To analyse the BSDE terms, we follow the approach of \citep[Section 2.2]{sun2016open}. Let $t \in [0,T]$, $x \in \R^n$, $u \in \HH^2([t, T]; \R^m)$, $\xi \in \HH^2([t, T]; \R^n)$ and $\eta \in \LL^2(\R^n, \F_T)$. Consider the FBSDE, 
\begin{align*}
dX^{x, u}_s & = \left(A_s X^{x, u}_s + B_s u_s \right) ds + \left( C_s X^{x, u}_s + D_s u_s \right) dW_s, \
X^{x, u}_t  = x, \\
dY^{\eta, \xi}_s & = -\left(A_s^\top Y^{\eta, \xi}_s + C_s^\top Z^{\eta, \xi}_s + \xi_s\right) ds + Z^{\eta, \xi}_s dW_s, \
Y^{\eta, \xi}_T  = \eta.
\end{align*}
By the variation of constants formula, $X^{x, u}$ admits the representation  for $s \in [t,T]$
\begin{align} \label{eq_rep}
X^{x, u}_s & = \Phi_s \Phi_t^{-1} x + (L_t u)_s, \qquad
X^{x, u}_T = \Phi_T \Phi_t^{-1} x + L_{t, T}u .
\end{align}
Applying It\^o's Lemma to $(X^{x, u})^\top Y^{\eta, \xi}$, taking the expectation, we have the Brownian motion term has expectation 0 as $X^{x, u} \in \s^2$ and $u, Z^{\eta, \xi} \in \HH^2$ \citep[Problem 2.10.7]{zhang2017backward}, and 
\begin{align*} 
 & \E\left[\left(\Phi_T \Phi_t^{-1} x + L_{t, T}u \right)^\top \eta - x^\top (Y^{\eta, \xi}_t)\right] \\
& = \E \left[ (X^{x, u}_T)^\top Y^{\eta, \xi}_T - (X^{x, u}_t)^\top Y^{\eta, \xi}_t\right] \\
& = \E \left[ \int_t^T \left(\left(A_s X^{x, u}_s + B_s u_s \right)^\top Y^{\eta, \xi}_s - (X^{x, u}_s)^\top\left(A_s^\top Y^{\eta, \xi}_s + C_s^\top Z^{\eta, \xi}_s + \xi_s\right) + \left(C_s X^{x, u}_s + D_s u_s\right)^\top Z^{\eta, \xi}_s \right) ds \right]\\
& = \E \left[ \int_t^T \left(u_s^\top\left(B_s^\top Y^{\eta, \xi}_s + D_s^\top Z^{\eta, \xi}_s\right) - \left((\Gamma_tx)_s + (L_t u)_s \right)^\top \xi_s\right) ds \right].
\end{align*}
Taking $x = 0$ yields
\begin{align} \label{eq_lin1} \begin{split}
 \E\left[\left(L_{t, T}u\right)^\top \eta \right]  = \E \left[ \int_t^T \left(u_s^\top\left(B_s^\top Y^{\eta, \xi}_s + D_s^\top Z^{\eta, \xi}_s\right)  - \left((L_t u)_s\right)^\top \xi_s\right) ds \right].
\end{split}
\end{align}
Furthermore, by definition of the adjoint operator, we have
\begin{align} \label{eq_lin2}
\E\left[(L_{t, T} u) ^\top \eta \right] = \E\left[\int_t^T u_s^\top (L^*_{t, T} \eta)_s ds \right], \\
\label{eq_lin3}
\E\left[\int_t^T (L_{t} u)_s^\top \xi_s ds \right] = \E\left[\int_t^T u_s^\top (L^*_{t} \xi)_s ds \right]. 
\end{align}
Substituting (\ref{eq_lin2}) and (\ref{eq_lin3}) into (\ref{eq_lin1}) yields
\begin{align} \label{eq_lin4} \begin{split}
& \E\left[\int_t^T u_s^\top (L^*_{t, T} \eta)_s ds \right] = \E \left[ \int_t^T \left(u_s^\top\left(B_s^\top Y^{\eta, \xi}_s + D_s^\top Z^{\eta, \xi}_s  - (L^*_{t} \xi)_s\right)\right) ds \right]. \end{split}
\end{align}
This yields an equation
 \[\left\langle u, B^\top Y^{\eta, \xi} + D^\top Z^{\eta, \xi}  - L_{t, T}^* \eta - L_t^* \xi \right\rangle_{\HH^2} = 0,\]
that is satisfied for all $u \in \HH^2(\R^m)$. The multiplier must therefore be 0 and we have for any $s \in [t, T]$
\begin{align*}
B_s^\top Y^{\eta, \xi}_s + D_s^\top Z^{\eta, \xi}_s & = (L_{t, T}^* \eta)_s + (L_t^* \xi)_s.
\end{align*}
Now we take $t = 0$, $x = x_0$, $u = u^k$, $\eta = \triangledown g(X^k_T)$, and $\xi = \partial_x f(X^k, u^k)$. Then $(X^{x, u}, Y^{\xi, \eta}, Z^{\xi, \eta} ) = (X^k, Y^k, Z^k)$ and the result follows.
\end{proof}

Now we bound the iterates of (\ref{update_control}) using this representation of the update formula. 

\begin{proof}[Proof of Theorem \ref{thm_diff}]
By Lemma \ref{lem_rep}, the update (\ref{update_control}) is given by
\begin{align} \begin{split} \label{update_control2}
u^{k+1}_t
& = \prox_{U}\left(u^k_t - \tau \left( \left[L_{0, T}^* \triangledown g\left(X^k_T\right)\right]_t + \left[L_0^* \partial_x f\left(X^k, u^k\right)\right]_t + \partial_u f_t\left(X^k_t, u^k_t\right)\right) \right).
\end{split}
\end{align}
Using the 1-Lipschitz property of $\prox_{U}$ we can bound the differences between iterates of (\ref{update_control2}) by
\begin{align*}
\left\|u^{k+1} - u^k\right\|_{\HH^2} & \leq \left\| u^k - u^{k-1} - \tau \left(I_1 + I_2 + I_3 + I_4
\right) \right\|_{\HH^2},
\end{align*}
where
\begin{align*}
I_1 & = L_{0, T}^* \left(\triangledown g\left( X^k_T\right) - \triangledown g\left(X^{k-1}_T\right)\right), \\
I_2 & = L_0^* \left(\partial_x f\left(X^k, u^k\right) - \partial_x f\left(X^{k-1}, u^{k - 1}\right)\right), \\
I_3 & =  \partial_u f^1\left(X^k, u^k\right) - \partial_u f^1\left(X^{k-1}, u^{k-1}\right), \\
I_4 & =  \partial_u f^2\left(u^k\right) - \partial_u f^2\left(u^{k-1}\right).
\end{align*}
We proceed with the analysis on a case-by cases basis under Assumption \ref{ass_strong}.

\textbf{Standard case.} We have
\begin{align*}
\left\|u^{k+1} - u^k\right\|_{\HH^2} & \leq \left\| u^k - u^{k-1} - \tau I_4
\right\|_{\HH^2} + \tau \left\|I_1 + I_2 + I_3 \right\|_{\HH^2} \\
& = \left(\left\| u^k - u^{k-1} \right\|^2_{\HH^2}  - 2\tau \left \langle u^k - u^{k-1}, I_4 \right \rangle_{\HH^2} + \tau^2 \left\|I_4\right\|^2_{\HH^2}
\right)^\half + \tau  \left\|I_1 + I_2 + I_3 \right\|_{\HH^2}.
\end{align*}
By \citep[Theorems 3.2.2 and 4.2.1]{zhang2017backward}, the Lipschitz property of the derivatives of the running and terminal cost, and the bounds on the linear operators, there exists $K \geq 0$, independent of $f^2$, such that
\begin{align*}
\left\|I_1 + I_2 + I_3\right\|_{\HH^2}\leq K  \left\| u^k - u^{k-1} \right\|_{\HH^2}.
\end{align*}
Since $f^2_t$ is $\mu$-strongly convex,  and $\partial_u f^2_t$ is Lipschitz continuous with Lipschitz constant $L_{f^2}$, both uniformly in $t$, then by 
\citep[Inequality (2.1.30)] {nesterov2003introductory}, we have $\mu\leq L_{f^2}$. 
Using this, and applying \citep[Theorem 2.1.12] {nesterov2003introductory} to $u \mapsto f^2_t(x, u)$, we bound
\begin{align*}
\left \langle u^k - u^{k-1}, I_4 \right \rangle_{\HH^2} 
&  \geq \frac{\mu L_{f^2}}{\mu + L_{f^2}} \left\|u^k - u^{k-1} \right\|_{\HH^2}^2 + \frac{1}{\mu + L_{f^2}}   \left\|\partial_u f^2\left( u^k\right) - \partial_u f^2\left(u^{k-1}\right) \right\|_{\HH^2}^2 \\
&  \geq \frac{\mu}{2} \left\|u^k - u^{k-1} \right\|_{\HH^2}^2 + \frac{1}{2L_{f^2}}   \left\|\partial_u f^2\left( u^k\right) - \partial_u f^2\left(u^{k-1}\right) \right\|_{\HH^2}^2.
\end{align*}
Hence, if $\tau<1/L_{f^2}$, then 
\begin{align*}
\left\|u^{k+1} - u^k\right\|_{\HH^2}
& \leq \left(\left(1 - \tau\mu\right) \left\| u^k - u^{k-1} \right\|^2_{\HH^2} + \tau \left(\tau - \frac{1}{L_{f^2}}  \right) \left\|I_4 \right\|_{\HH^2}^2
\right)^\half + K\tau  \left\| u^k - u^{k-1} \right\|_{\HH^2} \\
& \leq \left(1 - \tau \mu\right)^\half \left\| u^k - u^{k-1} \right\|_{\HH^2} + K\tau  \left\| u^k - u^{k-1} \right\|_{\HH^2} \\
& \leq \left(1 - \tau\left(\frac{\mu}{2} - K\right)\right) \left\| u^k - u^{k-1} \right\|_{\HH^2}
\end{align*}
as required.

\textbf{Singular case.} Similar to the previous case, we have the bounds
\begin{align*}
\left\|u^{k+1} - u^k\right\|_{\HH^2} & \leq \left\| u^k - u^{k-1} - \tau I_1
\right\|_{\HH^2} + \tau \left\|I_2 + I_3 + I_4 \right\|_{\HH^2} \\
& = \left(\left\| u^k - u^{k-1} \right\|^2_{\HH^2}  - 2\tau \left \langle u^k - u^{k-1}, I_1 \right \rangle_{\HH^2} + \tau^2 \left\|I_1\right\|^2_{\HH^2}
\right)^\half + \tau  \left\|I_2 + I_3 + I_4 \right\|_{\HH^2},
\end{align*}
and
\begin{align} \label{eq_singular_bound}
\left\|I_2 + I_3 + I_4\right\|_{\HH^2}\leq K  \left\| u^k - u^{k-1} \right\|_{\HH^2},
\end{align}
for some $K$ independent of $g$. Noting that $ X^k_T - X^{k-1}_T = L_{0, T}(u^k - u^{k-1})$ by \eqref{eq_rep}  and using the strong convexity of $g$ with \citep[Theorem 2.1.12] {nesterov2003introductory}, we have
\begin{align*}
 \left\langle u^k - u^{k-1}, I_1 \right\rangle_{\HH^2} 
& = \left\langle L_{0, T}\left(u^k - u^{k-1}\right), \triangledown g\left(X^k_T\right) - \triangledown g\left(X^{k-1}_T\right) \right\rangle_{\LL^2} \\
& = \bigg\langle X^k_T - X^{k-1}_T,\triangledown g\left(X^k_T\right) - \triangledown g\left(X^{k-1}_T\right) \bigg\rangle_{\LL^2} \\
& \geq \frac{\mu}{2} \left\|X^k_T - X^{k-1}_T \right\|_{\LL^2}^2 + \frac{1}{2 L_g}   \left\|\triangledown g\left(X^k_T\right) - \triangledown g\left(X^{k-1}_T\right)\right\|_{\LL^2}^2 \\
& = \frac{\mu}{2} \left\|L_{0, T}\left(u^k - u^{k-1}\right) \right\|_{\LL^2}^2 + \frac{1}{2 L_g}   \left\|\triangledown g\left(X^k_T\right) - \triangledown g\left(X^{k-1}_T\right)\right\|_{\LL^2}^2,
\end{align*}
where $L_g$ is the Lipschitz constant of $\triangledown g$. Hence, if $\tau \leq \frac{1}{ L_g \|L_{0, T}\|^2}$, then
\begin{align*}
& \left\| u^k - u^{k-1} \right\|_{\HH^2}^2 -2 \tau \left\langle u^k - u^{k-1}, I_1 \right\rangle_{\HH^2}  + \tau^2 \left\|I_1 \right\|_{\HH^2}^2 \\
&  \leq \left\| u^k - u^{k-1} \right\|_{\HH^2}^2 -\tau \mu\left\|L_{0, T}\left(u^k - u^{k-1}\right) \right\|_{\LL^2}^2  + \tau\left(  \tau\|L_{0, T}\|^2  - \frac{1}{L_g}  \right) \left\|\triangledown g\left(X^k_T\right) - \triangledown g\left(X^{k-1}_T\right)\right\|_{\LL^2}^2 \\
&  \leq \left\| u^k - u^{k-1} \right\|_{\HH^2}^2 - \tau \mu \left\|L_{0, T}\left(u^k - u^{k-1}\right) \right\|_{\LL^2}^2.
\end{align*}
 Combining with (\ref{eq_singular_bound}) gives 
\begin{align}
\left\| u^{k+1} - u^{k} \right\|_{\HH^2} & \leq 
 \left(\left\| u^k - u^{k-1} \right\|_{\HH^2}^2 - \tau \mu \left\|L_{0, T}\left(u^k - u^{k-1}\right) \right\|_{\LL^2}^2\right)^\half   + K \tau \left\| u^{k} - u^{k - 1} \right\|_{\HH^2}. \label{eq_new3}
\end{align}
We now need to handle the $L_{0, T}$ term.  We claim that there exists $\lambda > 0$ such that for any $u \in \HH^2(\R^m)$
\begin{align} \label{eq_sing}
\|L_{0, T} u\|^2_{\LL^2} \geq \lambda \|u\|^2_{\HH^2}.
\end{align}
If $\tau < \frac{1}{\mu \lambda}$, then taking $u = u^{k} - u^{k-1}$ and substituting into (\ref{eq_new3}) gives
\begin{align*}
\left\| u^{k+1} - u^{k} \right\|_{\HH^2}
&  \leq \left(\left(1 - \tau \mu\lambda\right)\left\| u^k - u^{k-1} \right\|_{\HH^2}^2  \right)^\half  + K \tau \left\| u^{k} - u^{k - 1} \right\|_{\HH^2} \\
&  \leq \left(1 - \tau\left(\frac{ \mu\lambda}{2} - K\right)\right)\left\| u^k - u^{k-1} \right\|_{\HH^2},   
\end{align*}
yielding the result. It remains to prove (\ref{eq_sing}). Let $u \in \HH^2(\R^m)$, then applying It\^o's Lemma to (\ref{eq_def_L}), the process $L_0 u \in \s^2(\R^n)$ satisfies
\begin{align*}
d\left|(L_0 u)_t\right|^2 & = \left(2(L_0u)_t^\top \left( A_t (L_0 u)_t + B_t u_t\right) + \left|C_t (L_0 u)_t + D_t u_t\right|^2 \right)dt + 2(L_0u)_t^\top\left(C_t (L_0 u)_t + D_t u_t\right) dW_t,
\end{align*}
with $\left|(L_0 u)_0\right|^2 = 0$. In particular, as $\|L_0u\|_{\s^2} < \infty$ and $\|CL_0u + Du\|_{\HH^2} < \infty$, the stochastic integral term of $\left|L_0 u\right|^2$ is a martingale so has expectation 0 \citep[Problem 2.10.7]{zhang2017backward}. Hence
\[\left\|L_{0, T}u \right\|_{\LL^2}^2 = \E\left[\left|(L_{0} u)_T\right|^2\right] = \E\left[ \int_0^T  \left( (L_0 u)_t^\top \A_t (L_0 u)_t + 2 (L_0 u)_t^\top \B_t u_t + u_t^\top \D_t u_t \right) dt \right],\]
where $\A, \B$ and $\D$ are defined in Assumption \ref{ass_strong}. We proceed by taking the cases outlined in this Assumption.

First, suppose that $\A_t$ is positive definite as in condition (i). Then $\A_t$ is invertible with symmetric inverse. For any $t \in [0,T]$, the function $x \mapsto x^\top \A_t x + x^\top \B_t u_t + u_t^\top \D_t u_t$ is minimised at $\hat{x} \defeq - \A_t^{-1} \B_t u_t$, with minimum value $u_t^\top \left(\D_t - \B_t^\top \A^{-1}_t \B_t\right) u_t$. Then assumption $\left(\D_t - \B_t^\top \A^{-1}_t \B_t\right) \succeq \delta \ind$ then implies
\begin{align*}
\left\|L_{0, T}u \right\|_{\LL^2}^2
&  \geq \E\left[ \int_0^T u_t^\top \left(\D_t - \B_t^\top \A^{-1}_t \B_t\right) u_t dt \right] \geq \delta \|u\|_{\HH^2}^2.
\end{align*}
Now, suppose $\A$ is  positive semi-definite. If $\B_t = 0$ for all $t$, we have  
\[\left\|L_{0, T}u \right\|_{\LL^2}^2 \geq \E\left[ \int_0^T  u_t^\top \D_t u_t \right] \geq \delta \|u\|^2_{\HH^2}.\]
If $\B_t \neq 0$ for some $t$, we have
\begin{align*}
\left\|L_{0, T}u \right\|_{\LL^2}^2 
& \geq \E\left[ \int_0^T  \left(  2 (L_0 u)_t^\top \B_t u_t + u_t^\top \D_t u_t \right) dt \right] \\
&  = \E\left[ \int_0^T  \left( \left|(L_0 u)_t + \B_t u_t\right|^2 -\left|(L_0 u)_t\right|^2 + u_t^\top\left( \D_t -\B_t^\top \B_t\right) u_t \right) dt \right] \\
&  \geq - T \left\|L_0 u\right\|^2_{\s^2} +  \E\left[ \int_0^T u_t^\top\left( \D_t -\B_t^\top \B_t\right) u_t  dt \right]. 
\end{align*}
By \citep[Theorem 3.2.2]{zhang2017backward} there exists $K$, depending on model parameters but independent of $T$, such that
$\|L_0u\|^2_{\s^2}   \leq K e^{K T}\|u\|^2_{\HH^2} $.
The assumption $\left(\D_t - \B_t^\top \B_t\right) \succeq \delta \ind$ then implies 
\begin{align} \label{eq_sing1}
\left\|L_{0, T}u \right\|_{\LL^2}^2 
& \geq  \left(\delta - K T e^{K T} \right) \|u\|^2_{\HH^2},
\end{align}
the coefficient on the right side remains positive for any $T$ satisfying $K T e^{K T} < \delta$.
\end{proof}

\begin{proof}[Proof of Theorem \ref{thm_convergence}]
Let $K_0$ be the constant from Theorem \ref{thm_diff}, that is independent of $\mu$, and take $\mu \geq \frac{1}{\rho}K_0$. Then define $\hat{c} \defeq 1 - \tau \left(\rho \mu - K_0\right)$. Taking $\tau$ sufficiently small as in Theorem \ref{thm_diff} gives $\hat{c} \in (0, 1)$ and $ \left\|u^{k+1} - u^k\right\|_{\HH^2} \leq \hat{c} \left\|u^k - u^{k-1}\right\|_{\HH^2}$ for all $k \in \N$. The sequence $\left(u^k\right)_{k\in\N}$ is then a Cauchy sequence in the Banach space $\HH^2(\R^m)$, so by Banach's fixed point Theorem, there exists $u^* \in \HH^2(\R^m)$ such that $\lim_{k \to \infty} \left\|u^k - u^*\right\|_{\HH^2} =  0$, with the convergence rate given by (\ref{eq_conv}). To show $u^* \in \U$, it remains to show that it takes values in $U$. This  follows from the proof that $u^*$ is a stationary point of $J$ under the control constraint.
Now, consider the FBSDE, for $s \in [0,T]$,
\begin{align*} 
dX^{*}_s & = \left(A_sX^{*}_s + B_s u^*_s\right) ds + \left(C_sX^{*}_s + D_s u^*_s\right) dW_s, \
X^{*}_0  = x_0, \\
dY^{*}_s & = -\left(A_s^\top Y^{*}_s + C_s^\top Z^{*}_s + \partial_x f_s\left(X^{*}_s, u^*_s\right)\right)ds + Z^{*}_s dW_s, \
Y^{*}_T = \triangledown g\left(X^{*}_T\right).
\end{align*}
By \citep[Theorems 3.2.2 and 4.2.1]{zhang2017backward}, we have $X^*, Y^*, Z^* \in \HH^2(\R^n)$ and furthermore there exists $K \geq 0$ such that for all  $k \in \N$,
\begin{align*}
\left\|X^{*} - X^{k}\right\|_{\HH^2} +
\left\|Y^{*} - Y^{k}\right\|_{\HH^2} +
 \left\|Z^{*} - Z^{k}\right\|_{\HH^2} & \leq K\left\|u^{*}
 - u^{k}\right\|_{\HH^2} \to 0,
 \end{align*}  
 where $X^k, Y^k, Z^k$ are given in (\ref{update_control}). In particular, using the Lipschitz properties of $\prox_{U}$ and $\partial_u H$, there exists $K \geq 0$ (that changes from line to line and is independent of $k$) such that for any $k \in \N$ 
\begin{align*}
& \left\|u^k - \prox_{U}\left(u^* - \tau \partial_u H\left({X}^*, u^*, {Y}^*, {Z}^*\right)\right)\right\|_{\HH^2}  \\
& = \left\|u^k - u^{k+1} +  \prox_{U}\left(u^k - \tau \partial_u H\left({X}^*, u^*, {Y}^*, {Z}^*\right)\right) - \prox_{U}\left(u^* - \tau \partial_u H\left({X}^*, u^*, {Y}^*, {Z}^*\right)\right)\right\|_{\HH^2} \\
& \leq K\left(\left\|u^k - u^{k+1}\right\|_{\HH^2} +  \left\|u^k - u^*\right\|_{\HH^2} + \left\|\partial_u H\left({X}^k, u^k, {Y}^k, {Z}^k\right) -  \partial_u H\left({X}^*, u^*, {Y}^*, {Z}^*\right)\right\|_{\HH^2} \right) \\
& \leq K\left(\left\|u^k - u^{k+1}\right\|_{\HH^2} +  \left\|u^k - u^*\right\|_{\HH^2} + \left\|X^k - {X}^*\right\|_{\HH^2}  + \left\|Y^k - {Y}^*\right\|_{\HH^2} + \left\|Z^k - {Z}^*\right\|_{\HH^2} \right) \\
& \to 0.
\end{align*}
Therefore, by uniqueness of limits we have
 \begin{align*}
  u^*  = \lim_{k \to \infty} u^k & = \prox_{U}\left(u^* - \tau \partial_u H\left({X}^*, u^*, {Y}^*, {Z}^*\right) \right) = \prox_{U}\left(u^* - \tau \triangledown J(u^*) \right),
 \end{align*}
 where the limit is taken in $\HH^2$. The control $u^*$ therefore satisfies $u \in \U$, proving the first statement. Furthermore, $u^*$ is a stationary point of $J$ under the control constraint by Theorem \ref{thm_stationary}, proving the second statement.
\end{proof}

\begin{proof}[Proof of Theorem \ref{thm_linear}] 
Fix $t \in [0,T]$ and $x \in \R^n$. Since the coefficients $B$, $D$, $R$,  $S$ in (\ref{eq_L_u}) are continuous in $t$, it suffices to show that $Y^{t,x,\phi}_t,\, Z_t^{t,x,\phi}$ are linear in $x$, and the linear coefficient is continuous in $t$. Substituting the definition of $\phi$ into (\ref{bsde}) gives us, for $s \in [t,T]$,
\begin{align*} 
dX^{t, x, \phi}_s & = \left(A_s + B_s \alpha_s \right)X^{t,x,\phi}_s  ds + \left(C_s + D_s \alpha_s\right)X^{t,x,\phi}_s  dW_s, \
X^{t,x,\phi}_t  = x, \\
dY^{t,x,\phi}_s & = -\left(A_s^\top Y^{t,x,\phi}_s + C_s^\top Z^{t,x,\phi}_s  + \left(Q_s + S_s^\top \alpha_s\right)X^{t,x,\phi}_s \right)ds + \left(Z^{t,x,\phi}_s\right)^\top dW_s, \
Y^{t,x,\phi}_T = GX^{t,x,\phi}_T.
\end{align*}
Due to the linear structure, we propose $v^{t, \phi}_s\left(\tilde{x}\right) = a_s \tilde{x} $ for $s \in [t,T]$ and $\tilde{x} \in \R^n$, for some $a\in \C^1\left([t,T];\R^{n \times n}\right)$. Substituting this into the HJB equation (\ref{eq_pde}), 
as $\tilde{x}$ is arbitrary, we have the ODE (\ref{update_a}).
Given this solution for $v^{\phi}$, we have the BSDE solutions
$Y_t^{t,x,\phi} = a_t x$ and $Z_t^{t,x,\phi} = a_t\left(C_t + D_t \alpha_t\right)x$, 
which are linear functions of $x$ with coefficients that are continuous in $t$.
\end{proof}

\bibliographystyle{apalike}

 \newcommand{\noop}[1]{}

\end{document}